\documentclass[12pt]{article}

\usepackage{geometry}
\usepackage{graphics}
\usepackage{amssymb}
\usepackage{amsbsy}
\usepackage{amsmath}
\usepackage{amsthm}
\usepackage{amsfonts}
\usepackage{extarrows}
\usepackage{dsfont}
\usepackage{float}
\usepackage[enableskew]{youngtab}
\usepackage{ytableau}
\usepackage[all]{xy}
\usepackage{setspace}
\usepackage{enumitem}

\usepackage{microtype}

\usepackage[square]{natbib}

\newtheorem{thm}{Theorem}[section]
\newtheorem{lem}[thm]{Lemma}
\newtheorem{prop}[thm]{Proposition}
\newtheorem{cor}[thm]{Corollary}

\newtheorem{remark}[thm]{Remark}
\newtheorem{dfn}[thm]{Definition}

\newtheorem{conjecture}[thm]{Conjecture}

\definecolor{grey}{gray}{0.41}

\newcommand{\mc}{\mathcal{C}}
\newcommand{\mg}{\mathcal{G}}
\newcommand{\mh}{\mathcal{H}}

\newcommand{\mP}{\mathcal{P}}
\newcommand{\mr}{\mathcal{R}}
\newcommand{\zz}{{\mathbb Z}}

\newcommand{\none}{n-1}

\newcommand{\HL}{\widetilde H_{\mu}(X;0,t)}

\newcommand{\HLdelta}{\widetilde H_{\delta}(X;0,t)}

\newcommand{\Des}{R_{\gamma,\delta}(X)}

\newcommand{\Macdelta}{\widetilde H_{\delta}(X;q,t)}
\newcommand{\MacStraight}{\widetilde H_{\mu}(X;q,t)}

\newcommand{\bnu}{{\boldsymbol{\nu}}}

\DeclareMathOperator{\rw}{rw}
\DeclareMathOperator{\Yam}{Yam}
\DeclareMathOperator{\SYam}{SYam}
\DeclareMathOperator{\Descent}{Des}
\DeclareMathOperator{\inv}{inv}
\DeclareMathOperator{\maj}{maj}
\DeclareMathOperator{\st}{st}
\DeclareMathOperator{\unst}{unst}
\DeclareMathOperator{\SYT}{SYT}

\DeclareMathOperator{\SF}{SF}

\numberwithin{equation}{section}

\begin{document}

\title{On the Schur expansion of Hall-Littlewood and related polynomials via Yamanouchi words}           

\author{Austin Roberts\thanks{Partially supported by DMS-1101017 from the NSF.}
\\ 
Department of Mathematics\\ Highline College\\
 Des Moines, WA 98198, USA\\
 }

\date{\today}

\maketitle


\begin{abstract}     
 This paper uses the theory of dual equivalence graphs to give explicit Schur expansions for several families of symmetric functions. We begin by giving a combinatorial definition of the modified Macdonald polynomials and modified Hall-Littlewood polynomials indexed by any diagram $\delta \subset {\mathbb Z} \times {\mathbb Z}$, written as $\Macdelta$ and $\HLdelta$, respectively. We then give an explicit Schur expansion of $\HLdelta$ as a sum over a subset of the Yamanouchi words, as opposed to the expansion using the charge statistic given in 1978 by Lascoux and Sch\"{u}ztenberger. We further define the symmetric function $\Des$ as a refinement of $\HLdelta$ and similarly describe its Schur expansion. We then analyze $\Des$ to determine the leading term of its Schur expansion. We also provide a conjecture towards the Schur expansion of $\Macdelta$. To gain these results, we use a construction from the 2007 work of Sami Assaf to associate each Macdonald polynomial with a  signed colored graph $\mh_\delta$. In the case where a subgraph of $\mh_\delta$ is a dual equivalence graph, we provide the Schur expansion of its associated symmetric function, yielding several corollaries. 
 \end{abstract}

\pagebreak




\section{Introduction}\label{intro} 

Adriano Garsia posed the question: when can the modified Hall-Littlewood polynomial $\HL$ be expanded in terms of Schur functions as a specific sum over Yamanouchi words, and is there a natural way to determine the expansion when it is not? At the time, he had already realized that  the expansion he had in mind does not apply when $\HL$ is indexed by the partition shape $\mu=(3,3,3)$. The results of this paper are in direct response to Garsia's question. In fact, the results we found proved to be more general than the question as originally posed.

In this paper, we will concentrate on three main families of polynomials, the modified Macdonald polynomials $\MacStraight$, the modified Hall-Littlewood polynomials $\HL$, and a refinement of the modified Hall-Littlewood polynomials, which we denote $\Des$. 
The Macdonald polynomials were introduced in \citep{Mac} and are often defined as the set of $q,t$-symmetric functions that satisfy certain orthogonality and triangularity conditions. 
The modified Macdonald polynomials were shown to be Schur positive by Mark Haiman via representation-theoretic and geometric means in \citep{Hai01}. In line with recent combinatorial work in \citep{HHL} and \citep{assaf2015dual}, we choose to work with the modified Macdonald polynomials, though it is relatively straightforward to transition to other forms of Macdonald polynomials.
%
%
%
Macdonald polynomials also specialize to several well-known symmetric functions, including Hall-Littlewood polynomials and Jack polynomials. While combinatorial descriptions of the Schur expansions of specific families of Macdonald polynomials can be found in \citep{HHL} (which drew on the earlier work in \citep{CL} and \citep{vanL}), \citep{Fishel}, \citep{Zab}, \citep{ZabrockiSpecialCase}, \citep{LM}, 
and \citep{DEG+LLT}, an explicit combinatorial description remains elusive outside of special cases. For details on the combinatorics of Macdonald polynomials, see \citep{Haglund}.

As just noted, Macdonald polynomials specialize to Hall-Littlewood polynomials, specifically by letting $q=0$. Hall-Littlewood polynomials, in turn, specialize to the Schur functions, Schur Q-functions, and the monomial symmetric functions. They were first studied by Philip Hall in relation to the Hall algebra in \citep{hall1957algebra} and later by D.E. Littlewood in \citep{littlewood1961certain}. It should be noted that the earliest known work on Hall-Littlewood polynomials actually dates back to the lectures of Ernst Steinitz in \citep{steinitz1901theorie}. 
 Hall-Littlewood polynomials have proven to be a rich mathematical topic, with recent combinatorial work including (but certainly not limited to) \citep{nakayashiki1997kostka}, \citep{Carbonara:1998:CII:301719.301728}, \citep{dalal2012abc},and \citep{Loehr:2013:TMS:2532882.2533254}. Expanding modified Hall-Littlewood polynomials into Schur functions can be achieved via the charge statistic, as found in \citep{lascoux1978conjecture}, though we will present a new expansion  in this paper as a sum over a subset of the Yamanouchi words. There is currently no simple method of deriving one expansion from the other --- though the discovery of one would certainly be of interest. This new result is promising both because it allows for a greater use of the known machinery of Yamanouchi words to be applied to Hall-Littlewood polynomials, as in Part 1 of Remark~\ref{yam tree} and Remark~\ref{first 3 rows}, and because it generalizes the usual setting of partition shapes to the more general $\delta \subset \zz \times \zz$. Equivalently, our result gives a new combinatorial rule for the coefficients of the Kostka-Foulkes polynomial in one variable $t$. For more background on the topic of Hall-Littlewood polynomials, see \citep{Macdonald}.

We use the statistics defined in \citep{HHL} to generalize the definition of the modified Macdonald polynomials $\MacStraight$ and the modified Hall-Littlewood polynomials $\HL$ to any diagram $\delta \subset \zz\times \zz$, giving the functions $\Macdelta$ and $\HLdelta$. We may then write $\HLdelta$ in terms of the refinement polynomials $\Des$, defined via row reading words of fillings of $\delta$ with a fixed descent set $\gamma\subset \delta$. We will discuss these polynomials in the general context of diagrams, rather than the more traditional setting of partitions. We may then write the main theorem of this paper as follows.

\begin{thm}\label{Hall Yam classification}
If $\gamma$ and $\delta$ are any diagrams such that $\gamma \subset \delta$, then
\begin{align}\label{HL Yam equation}
\HLdelta = \sum_{\lambda \vdash |\delta|} \;
\sum_{{w \in \Yam_{\delta}(\lambda)}\atop \inv_\delta(w)=0 }t^{\maj_\delta(w)} s_\lambda,&&
\Des =&\sum_{\lambda\vdash |\delta|}
 \sum_{{w \in \Yam_{\delta}(\lambda)}\atop {\inv_\delta(w)=0\atop \Descent_\delta(w)=\gamma }} s_\lambda.
\end{align}
\end{thm}

\noindent Here, $\Yam_\delta(\lambda)$ is the subset of the Yamanouchi words with content $\lambda$ whose elements, when thought of as row reading words of a filling of the diagram $\delta$, never have the $j^{th}$ from last $i$ in the same pistol of $\delta$ as the $j+1^{th}$ from last $i+1$. The above definitions and notation will be given a more thorough treatment in Section~\ref{Preliminaries}.


The main tool used in the proof of Theorem~\ref{Hall Yam classification} is the theory of dual equivalence graphs. Dual equivalence has its roots in the work of Sch\"{u}tzenberger in \citep{schutzenberger579correspondance},  Mark Haiman in \citep{Haiman}, and Donald Knuth in \citep{Knuth}. Sami Assaf introduced the theory of dual equivalence graphs in her Ph.D. dissertation \citep{Assaf07} and subsequent paper \citep{assaf2015dual}. The theory was further advanced by the author in \citep{DEG+LLT}, from which we will derive the definition of dual equivalence graph used in this paper.  In these papers, a dual equivalence graph is associated with a symmetric function so that each component of the graph corresponds to a single Schur function. Thus, the Schur expansion of said symmetric function is described by a sum over the set of components of the graph. Specifically, each component of a dual equivalence graph is isomorphic to a unique standard dual equivalence graph $\mg_\lambda$, which in turn corresponds to the Schur function $s_\lambda$. Variations of dual equivalence graphs have also been given for $k$-Schur functions in \citep{AB}, for the product of a Schubert polynomial with a Schur polynomial in \citep{ABS}, and for shifted tableaux in relation to the type $B$ Lie group in \citep{Billey2014Coxeter} and \citep{assaf2014shifted}. Dual equivalence graphs were also connected to Kazhdan-Lusztig polynomials and $W$-graphs by Michael Chmutov in \citep{Chmutov}.

This paper will focus on dual equivalence graphs that emerge as components of a larger family of graphs. The involution $D_i^{\delta}\colon S_n \to S_n$ was first introduced in \citep{Assaf07} and can be used 
 to define the edge sets of a signed colored graph $\mh_{\delta}$ with vertex set $S_n$ and vertices labeled by the signature function $\sigma$, which is defined via the inverse descent sets of permutations. We may then associate $\HL$ and $\Des$ with subgraphs of $\mh_\delta$. We show that these two subgraphs are dual equivalence graphs in Theorem~\ref{Halls are DEGs}. The main contribution of this paper to the theory of dual equivalence graphs can then be stated in the following theorem.

\begin{thm}\label{yam decomposition of graphs}
Let $\delta$ be a diagram of size $n$, and let $\mg=(V, \sigma, E)$ be a dual equivalence graph such that $\mg$ is a component of $\mh_{\delta}$ and $\mg\cong\mg_\lambda$. Then there is a unique vertex of $V$ in $\SYam_\delta(\lambda)$, and $V\cap\SYam_\delta(\mu)=\emptyset$ for all $\mu\neq \lambda$.
\end{thm}

\noindent Here, $\SYam_\delta(\lambda)$ is the set of permutations resulting from standardizing the words in $\Yam_\delta(\lambda)$.

This paper is organized as follows. We begin with the necessary material from the literature in Section~\ref{Preliminaries}, discussing tableaux, symmetric functions, and dual equivalence graphs. In Section~\ref{DEGs of tau}, we give a classification of which connected components of $\mh_{\delta}$ are dual equivalence graphs in Lemma~\ref{DEG by pattern}, and use this to prove that the signed colored graphs associated with $\Des$ and $\HLdelta$ are dual equivalence graphs. We then prove Theorem~\ref{yam decomposition of graphs}, followed by Theorem~\ref{Hall Yam classification}, as well as some related results. Next, Conjecture~\ref{f tau} gives a possible direction towards the Schur expansion of Macdonald polynomials. 
 Section~\ref{Applications to Symmetric Functions} is dedicated to further analysis of $\MacStraight$ and $\HL$. After classifying when $\MacStraight$ and $\HL$ expand via Yamanouchi words in Corollary~\ref{Hall Yam cor classification} and Proposition~\ref{Mac Yam equation}, we then end by classifying when $\Des=0$ in Proposition~\ref{Des zero} and giving a description of the leading term in the Schur expansion of $\Des$ in Proposition~\ref{Des leading term}. 



\section{Preliminaries}\label{Preliminaries}

\subsection{Tableaux and Permutations}\label{Tableaux}
By a \emph{diagram} $\delta$, we mean a finite subset of $\zz \times \zz$. We let $|\delta|$ denote the size of this subset. By reflecting a diagram $\delta$ over the line $x=y$ in the Cartesian plane, we may obtain the \emph{conjugate} diagram, denoted $\delta^\prime$. A \emph{partition} $\lambda$ is a weakly decreasing finite  sequence of nonnegative integers $\lambda_1 \geq \ldots \geq \lambda_k  \geq 0$. We write $|\lambda|=n$ or $\lambda\vdash n$ if $\sum \lambda_i = n$. We will give the diagram of a partition in French notation by drawing left justified rows of boxes, where $\lambda_i$ is the number of boxes in the $i^{th}$ row, from bottom to top, with bottom left cell at the origin, as in the left diagram of Figure \ref{partitions}. Any diagram that arises from a partition in this fashion is said to have \emph{partition shape}. Notice that conjugating a partition shape provides another partition shape.

\begin{figure}[h]
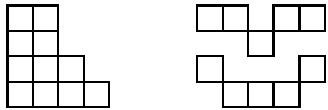

   \begin{center}
   \ytableausetup{smalltableaux}
    \ytableaushort{{}{},{}{},{}{}{},{}{}{}{}} \hspace{.4in}
   \ytableaushort{{}{}\none{}{},\none\none{},{}\none\none\none{},\none{}{}{}} 
   \end{center}
  \caption{The diagrams for (4,3,2,2) and an arbitrary $\delta\subset \zz\times\zz$.}
 \label{partitions}
\end{figure}

A \emph{filling} is a function that takes each cell of a diagram $\delta$ to a positive integer. We express a filling visually by writing the value assigned to a cell inside of the cell. A \emph{standard} filling uses each value in some $[n]=\{1,\ldots,n\}$ exactly once. Here, we say that $T$ is a standard filling of $\delta$, and we define $\SF(\delta)$ as the set of standard fillings of $\delta$. 
A \emph{standard Young tableau} is a standard filling in which all values are required to be increasing up columns and across rows from left to right. The set of all standard Young tableaux on diagrams of partition shape $\lambda$ is denoted by SYT($\lambda$), and the union of $\SYT(\lambda)$ over all $\lambda\vdash n$ is denoted $\SYT(n)$. 
For more information, see \cite[Part I]{Fulton}, \cite[Ch. 3]{Sagan}, or \cite[Ch. 7]{Stanley2}.

Define the \emph{row reading word} of a filling $T$, denoted $\rw(T)$, by reading across rows from left to right, starting with the top row and working down, as in Figure~\ref{reading word}. The row reading word of a standard filling is necessarily a permutation. For ease of reading, we will use $\pi$ for permutations and use $w$ in the more general context of words. Given a word $w$ and diagram $\delta$, both of size $n$, we let $T_\delta(w)$ denote the filling of diagram $\delta$ with reading word $w$.

\begin{figure}[h]
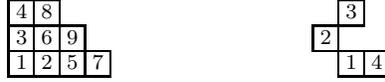

   \begin{center} 
   \ytableausetup{smalltableaux}
\ytableaushort{48,369,1257}
\hspace{1in}
\begin{ytableau}
\none& 3  \\
2 \\
\none& 1 &4
\end{ytableau}
   \end{center}
  \caption{On the left, a standard Young tableau with row reading word 483691257. On the right, a standard filling with row reading word 3214.}
 \label{reading word}
\end{figure}

By a \emph{pistol} of a diagram $\delta$, we mean a set of cells, in row reading order, between some cell $c$ and the position one row below $c$, inclusive. By a pistol of a filling $T$, we mean a pistol of the diagram of $T$. If $|\delta|=n$, we may associate each cell with a number in $[n]$ in row reading order. In turn, this associates each pistol of $\delta$ with an interval $I\subset [n]$. In this case, we say that $I$ is a pistol of $\delta$ and any collection of numbers in $I$ is said to be \emph{$\delta$-pistoled}, as shown in Figure~\ref{pistols}. 
In particular, given a word $w$ of length $|\delta|$, any collection of indices corresponding to cells in a pistol of $T_\delta(w)$ is 
are $\delta$-pistoled. As an example, using the standard filling on the right of Figure~\ref{pistols}, $w=3214$ and the sets $\{1, 2\}$ and $\{3, 4\}$ are $\delta$-pistoled. Equivalently, we could say that the indices of $3$ and $2$ are $\delta$-pistoled, as are the indices of $1$ and $4$.

\begin{figure}[h]
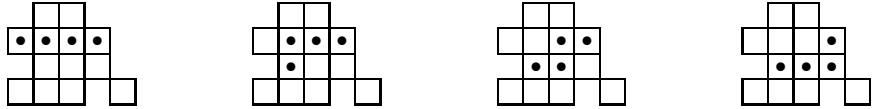

   \begin{center} 
   \ytableausetup{smalltableaux}
\begin{ytableau}
\none&{}&{} \\
\bullet&\bullet&\bullet& \bullet \\
\none&{}&{}& {} \\
{}&{}&{}&\none&{}
\end{ytableau}
\hspace{.5in}
\begin{ytableau}
\none&{}&{} \\
{}&\bullet&\bullet& \bullet \\
\none&\bullet&{}& {} \\
{}&{}&{}&\none&{} 
\end{ytableau}
\hspace{.5in}
\begin{ytableau}
\none&{}{}&{}{} \\
{}&{}&\bullet& \bullet \\
\none&\bullet&\bullet& {} \\
{}&{}&{}&\none&{}
\end{ytableau}
\hspace{.5in}
\begin{ytableau}
\none&{}{}&{}{} \\
{}&{}&{}& \bullet \\
\none&\bullet&\bullet& \bullet \\
{}&{}&{}&\none&{}
\end{ytableau}
   \end{center}
  \caption{Four pistols filled with bullets, corresponding to the four sets $\{3, 4, 5, 6\}, \{4,5,6,7\}, \{5, 6, 7, 8\}, \{6, 7, 8, 9\}\subset [13]$, respectively.  All of these sets are $\delta$-pistoled, as are all of their subsets.}
 \label{content}\label{pistols}
\end{figure}

Given a permutation $\pi$ in one-line notation, the \emph{signature} of $\pi$, denoted $\sigma(\pi)$ is a string of 1's and $-1$'s, or +'s and $-$'s for short, where there is a $+$ in the $i^{th}$ position if and only if $i$ comes before $i+1$ in $\pi$. Notice that a permutation is one entry longer than its signature.  The signature of a standard filling $T$ is defined as $\sigma(T):=\sigma(\rw(T))$. As an example, the signatures of the standard fillings in Figure \ref{reading word} are $+--+-+-+$ and $--+$, respectively.

For $I\subset [n]$ and $\pi \in S_n$, let $\pi|_I$ be the word given by reading the values of $I$ in the order they appear in $\pi$. For any such $I$, $\pi|_I$ is referred to as a \emph{subword} of $\pi$.  Given any standard filling $T$ of size $n$, we let $T|_I$ denote the filling that results from removing all cells of $T$ with values not in $I$, as shown in Figure~\ref{restriction figure}. Given a set of permutations $S\subset S_n$, we let $S|_I=\{\pi|_I\colon \pi \in S\}$.

\begin{figure}[h]
   \ytableausetup{smalltableaux, aligntableaux=center}
\begin{align*}
\pi&=483691257 &T&= \ytableaushort{48,369,1257}\\
\pi|_{\{6, 7, 8, 9\}} &= 8697 &T\big|_{\{6, 7, 8, 9\}} &= \ytableaushort{\none8,\none69,\none\none\none7}
\end{align*}
\caption{At left, a word $\pi$ above $\pi|_I$, where $I=\{6, 7, 8, 9\}$. At right, a standard Young tableau $T$ above $T|_I$ with the same $I$.}\label{restriction figure}
\end{figure}

We may \emph{standardize} a word $w$ of length $n$ by replacing the values in $w$ with the values in $[n]$ as follows. If there are $k$ 1's, replace the 1's in $w$ with the values 1 through $k$ from left to right. Then replace the 2's in $w$ in a similar fashion, replacing the first 2 with the value $k+1$. Repeat this recursively until $w$ has been replaced by a permutation, which we will denote $\st(w)$. Notice $\st(\pi)=\pi$ for all $\pi \in S_n$.

We may \emph{unstandardize} a permutation $\pi$ by replacing each value $i$ with 1 plus the number of $-1$'s in $\sigma(\pi|_{[i-1]})$, resulting in a word that we will denote $\unst(\pi)$. That is, $i$ and $i+1$ are taken to the same value if $i$ occurs before $i+1$. Otherwise, $i+1$ is taken to the value that is one larger than that of $i$. We may unstandardize a word $w$ by letting $\unst(w)=\unst(\st(w))$. Notice that $\st(\unst(\pi))=\pi$ and $\unst(\unst(w))=\unst(w)$. See Figure~\ref{unst} for an example.

\subsection{The Robinson-Schensted-Knuth correspondence}
The Robinson-Schensted-Knuth (R-S-K) correspondence gives a bijection between permutations in $S_n$ and pairs of standard Young tableaux $(P,Q)$, where $P$ and $Q$ have the same shape $\lambda \vdash n$. Here, $P$ is called the \emph{insertion tableau} and $Q$ is called the \emph{recording tableau}.  $P(\pi)$ Let $P\colon S_n\rightarrow \SYT(n)$ be the function taking a permutation to its insertion tableau. Given any $T\in\SYT(n)$, the set of permutations $\pi$ such that $P(\pi)=T$ is called a \emph{Knuth class}. More generally, the R-S-K correspondence can be defined for all words $w$ by considering each occurrence of a value $w_i$ to be less than any occurrence of the value later in the $w$. In this way, the $P$ tableaux is allowed to have repeated entries, though this paper will largely focus on the permutation case. We will assume a solid understanding of the R-S-K correspondence, but the reader can find more details on the many properties of the R-S-K correspondence in \cite{Fulton} or \cite{Sagan}.

One common method of computing $(P(\pi), Q(\pi))$ is to use the ``row bumping algorithm,'' which builds $P(\pi)$ on cell at a time by ``inserting'' the successive values of $\pi$ while recording the order cells are added in $Q(\pi)$.  It follows from the fact that the row bumping algorithm only considers relative order, with the latter of two equal numbers viewed as larger, that standardization preserves recording tableaux. Similarly, standardizing a word has the effect of standardizing its insertion tableaux via its row reading word.  It also follows from the row bumping algorithm, specifically the fact that the cell containing $n$ during the insertion process has no effect on the paths of smaller values, that $P(\pi|_{[n-1]})$ is the result of removing the cell containing $n$ from $P(\pi)$.

For $\lambda \vdash n$, let $U_\lambda$ be the standard Young tableau of shape $\lambda$ given by placing the numbers in $[n]$ in order across the first row of $\lambda$, then across the second row, and so on. Now define $\SYam(\lambda):= \{ \pi\in S_n \colon P(\pi) = U_\lambda \}$, and call this set of permutations the \emph{standardized Yamanouchi words of shape $\lambda$}. 
Let $\Yam(\lambda)$ denote the set of all words $w$ of length $n$ such that there are never more $i+1$'s than $i$'s while reading from right to left with the further requirement that $i$ occurs $\lambda_i$ times in $w$, as demonstrated in Figure~\ref{unst}. Any such word is called a \emph{Yamanouchi word}.  Notice that $\unst(w)=w$ for all $w\in\Yam(\lambda)$. It follows directly from the row bumping algorithm that if $\Yam(\lambda)$ is precisely the set of words $w$ such that $P(w)$ is the tableau with $\lambda_i$ many $i$'s in the $i^{th}$ row. As mentioned above, standardizing $w$ has the effect of standardizing this tableaux via the row reading word, so $\SYam(\lambda)=\{\st(w) \colon w \in \Yam(\lambda)\}$. Similarly, $\Yam(\lambda)=\{ \unst(\pi) \colon \pi \in \SYam(\lambda)\}$.  Finally, if $\pi \in \SYam(\lambda)$, where $\lambda=(\lambda_1, \cdots, \lambda_k)$, then $P(\pi|_{[n-1]})$ is the result of removing $n$ from $U_\lambda$, and so $\pi|_{[n-1]}\in\SYam((\lambda_1, \cdots, \lambda_{k-1},\lambda_k -1))$.

\begin{figure}[h]
\begin{center}
 \begin{minipage}{3in}
 \begin{eqnarray*}
       v &=& 17215 \\ 
\st(v) &=& 15324 \notin \SYam((2,2,1)) \\
 \unst(v)&=& 13212 \notin \Yam((2,2,1))\\
 P(\st(v))&=&\ytableaushort{5,3,124}
\end{eqnarray*} \end{minipage} \quad
 \begin{minipage}{3in}
 \begin{eqnarray*}
       w &=& 38631242 \\ 
\st(w) &=& 48751263 \in \SYam((3,3,1,1)) \\
 \unst(w)&=& 24321121 \in \Yam((3,3,1,1))\\
 P(\st(w))&=&\ytableaushort{8,7,456,123}
\end{eqnarray*} \end{minipage}\end{center}
\caption{Standardizing and unstandardizing the words $v$ and $w$. In the case of $w$, the result is a standardized Yamanouchi word and a Yamanouchi word, respectively, as demonstrated by the insertion tableaux of their standardizations at bottom.} \label{unst}
\end{figure}

The following definitions are referenced in the statement of Theorems~\ref{Hall Yam classification} and \ref{yam decomposition of graphs}.

\begin{dfn}~\label{jammed}
\textup{Let $\delta$ be any diagram, and let $w$ be any Yamanouchi word of length $|\delta|$. We say that $w$ \emph{jams} $\delta$ if there exists some $i$ and some $j$ such that indices of the $j^{th}$ from last $i$ in $w$ and the $j+1^{th}$ from last $i+1$ in $w$ are $\delta$-pistoled. A standardized Yamanouchi word is said to jam $\delta$ if $\unst(\pi)=w$ jams $\delta$. In the context of having such a $w$, we refer to the index of $w$ containing the $j^{th}$ from last $i$ as \emph{jamming} a pistol of $\delta$.}
\end{dfn}

We may then define
\begin{align} \begin{split}
\Yam_\delta(\lambda)&:=\{w\in\Yam(\lambda)\colon w\textup{ does not jam $\delta$}\},\\
\SYam_\delta(\lambda)&:=\{\pi\in\SYam(\lambda)\colon \pi \textup{ does not jam $\delta$}\},
\end{split}
\end{align}

\noindent with examples of each set given in Figure~\ref{jamming}.

\begin{figure}[h]
\ytableausetup{aligntableaux=center}
\begin{eqnarray*}
 \Yam_{(3,3)}(2,2,2)=
 \left\{\rw\biggl(\ytableaushort{321,321}\biggr), \rw\biggl(\ytableaushort{323,121}\biggr)\right\},&
 \rw\biggl(\ytableaushort{33{\bf2},12{\bf1}}\biggr), \rw\biggl(\ytableaushort{{\bf3}23,{\bf2}11}\biggr) \notin \Yam_{(3,3)}(2,2,2)
 \\
 %
\SYam_{(3,3)}(2,2,2)=\left\{\rw\biggl(\ytableaushort{531,642}\biggr), \rw\biggl(\ytableaushort{536,142} \biggr)\right\},&
\rw\biggl(\ytableaushort{56{\bf3},14{\bf2}}\biggr), \rw\biggl(\ytableaushort{{\bf5}36,{\bf4}12}\biggr)\notin\SYam_{(3,3)}(2,2,2)\\
\end{eqnarray*} 
\caption{At left, the sets $\Yam_{(3,3)}(2,2,2)$ and $\SYam_{(3,3)}(2,2,2)$. At right, examples of words in $\Yam(2,2,2)$ and $\SYam(2,2,2)$ that jam $(3,3)$. The bottom row is achieved by standardizing the top row.} \label{jamming}
\end{figure}

\begin{remark}\label{yam tree}
\textup{
\begin{enumerate}
\item One method for listing Yamanouchi words is to  begin with the number 1 and add numbers to the left of it, as in the description of Yamanouchi words above. The further condition that a word not jam $\delta$ simply means that upon adding the $j+1^{th}$ $i+1$, it needs to be checked that this $i+1$ is not in a pistol with the $j^{th}$ $i$. That is, the process of generating Yam$_\delta(\lambda)$ is readily integrated into the procedure for finding Yam$(\lambda)$. 
\item For the reader that prefers permutations, we may describe $\SYam_\delta(\lambda)$ as follows. Consider the result of right justifying $U_\lambda$, and let $A_\lambda$ be the set of pairs of values in cells that are touching on a southeasterly diagonal. For any $\pi\in\SYam(\lambda)$, consider $T_\delta(\pi)$. Then $\pi \in \SYam_\delta(\lambda)$ if and only if no pairs in $A_\lambda$ are in a pistol of $T_\delta(\pi)$. In this way, we have encoded which pairs correspond to the $j^{th}$ from last $i$ and the $j+1^{th}$ from last $i+1$ in $\unst(\pi)$ into the set $A_\lambda$, and then look for these pairs in the pistols of $T_\delta(\pi)$.  See Figure~\ref{perm jam} for an example.
\begin{figure}[h]
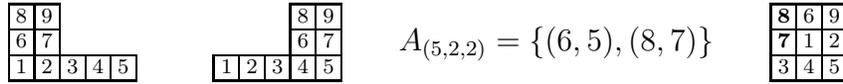

  \ytableausetup{aligntableaux=center}
  \[ \ytableaushort{89, 67, 12345}  \hspace{.4in}
   \ytableaushort{\none \none \none 89, \none \none \none 67, 12345} \hspace{.3in} 
   A_{(5,2,2)} =\{(6, 5), (8, 7)\} \hspace{.3in} 
   \ytableaushort{{\bf 8}69, {\bf 7}12,345}
  \]
\caption{From the left, $U_{(5,2,2)}$, followed by the result of right justifying, followed by $A_{(5,2,2)}$, followed by a standard filling $T$ of partition shape $\mu=(3,3,3)$ such that $\rw(T)\in \SYam((5,2,2))$ and the pair $(8,7)\in A_{(5,2,2)}$ is in a pistol of $\mu$. Thus $\rw(T)\notin \SYam_\mu((5,2,2))$.} \label{perm jam}
\end{figure}
\item The set of permutations in $\SYam(\lambda)$ is a Knuth class. The set $\SYam_\delta(\lambda)$ is necessarily a subset of this class, and so can be expressed via some set of recording tableaux of a given partition shape. Finding a more explicit way of generating all such recording tableaux remains an interesting open problem.
\end{enumerate}
 }
\end{remark}

\subsection{Symmetric Functions}\label{Symmetric Functions}

We briefly recall the interplay between standard fillings, quasisymmetric functions, Schur functions, Macdonald polynomials, and Hall-Littlewood polynomials. All of which have rich connections to the theory of symmetric functions. The reader may also refer to \cite[Ch. 7]{Stanley2}, \cite[Part I]{Fulton}, or \cite[Ch.~4]{Sagan}.

\begin{dfn}
\textup{
Given any signature $\sigma \in \{\pm 1\}^{n-1}$, define the \emph{fundamental quasisymmetric function $F_\sigma (X)$} $\in \mathds{Z}[x_1, x_2, \ldots]$ by}
\[
F_\sigma(X):= \sum_{{i_1 \leq \ldots \leq i_n \atop i_j = i_{j+1} \Rightarrow \sigma_j = +1}} x_{i_1}\cdots x_{i_n}.
\]
\end{dfn}

We may now use the previous definition to define the Schur functions, relying on a result of Ira Gessel. While it is not the standard definition, it is the most functional for our purposes.

\begin{dfn}\label{Ges}
\textup{{\bf \citep{Gessel}} 
\, Given any partition $\lambda$, define 
}
\begin{equation}
s_{\lambda}:= \sum_{T \in \SYT(\lambda)} F_{\sigma(T)}(X),
\end{equation}
where $s_{\lambda}$ is a Schur function of shape $\lambda$.
\end{dfn} 

In order to define the modified Macdonald polynomials and the modified Hall-Littlewood polynomials, we will first need to define some statistics, relying on the results in \citep{HHL} for our definitions. Let $T$ be a filling of a diagram $\delta$. Given a cell $c\in \delta$, let $T(c)$ denote the value of $T$ in cell $c$. A \emph{descent} of $T$ is a cell $c$ in $\delta$ such that $T(c)>T(d)$, where $d$ is the cell directly below and adjacent to $c$. Here, $c$ must have a cell directly below it in $\delta$ in order to be a descent.  We denote the set of descents of $T$ as $\Descent(T)$. As an example, the filling on the left of Figure~\ref{reading word} has descents in the cells containing 3, 4, 6, 8,  and 9, while the filling on the right has no descents. 
Given a cell $c$ of $\delta$, define leg$(c)$ as the number of cells in $\delta$ strictly above and in the same column as $c$. Letting $w=\rw(T)$, we may then define
\begin{equation}
\maj_\delta(w):=\maj(T):=\maj_\delta(\Descent(T)):=\sum_{c\in \Descent(T)} (1+\operatorname{leg}(c)).
\end{equation}


Let $c$, $d$, and $e$ be cells of $\delta$ in row reading order. Then $c$, $d$, and $e$ form a \emph{triple} if $c$ and $d$ are in the same row and $e$ is the cell immediately below $c$, as in Figure~\ref{triple}. If, in addition, $T$ is a filling of the diagram $\delta$, then $c$, $d$, and $e$ form an \emph{inversion triple} of $T$ if $T(e)<T(d)<T(c)$, $T(c)\leq T(e)< T(d)$, or $T(d)< T(c)\leq T(e)$. As a mnemonic, in each set of inequalities, the three cells are presented in a counterclockwise order. 
If $c$ is not in $\delta$, then the remaining two cells form an \emph{inversion pair} of $T$ if $T(d)>T(e)$. 
If $e$ is not in $\delta$, then the remaining two cells form an inversion pair of $T$ if $T(c)>T(d)$. See Figure~\ref{triple} for an example of each of these types of inversions.
By letting $w=\rw(T)$, we may now define the final statistic as
\begin{equation}
\inv_\delta(w):=\inv(T):=
|\{\textup{inversion triples of $T$}\}|+|\{\textup{inversion pairs of $T$}\}|.
\end{equation}

\begin{figure}[H]
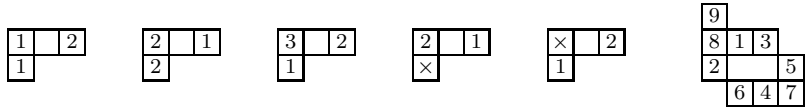

\ytableausetup{aligntableaux=center, smalltableaux}
  \[
   \ytableaushort{1{}2, 1} \hspace{.3in} 
   \ytableaushort{2{}1, 2} \hspace{.3in} 
   \ytableaushort{3{}2, 1} \hspace{.3in} 
   \ytableaushort{2{}1, \times}  \hspace{.3in}
   \ytableaushort{\times{}2, 1} \hspace{.4in} 
   \begin{ytableau}
9\\
8& 1& 3  \\
2&\none&\none&5 \\
\none& 6 &4&7
\end{ytableau}
     \]
\caption{From left, three inversion triples, and then two inversion pairs, where $\times$ denotes the lack of a cell. At right, a filling $T$ that has one inversion triple with values in $\{2, 3, 8\}$, inversion pairs with values in $\{2, 3\}, \{4, 5\}$,  and $\{4, 6\}$, and descents in the cells with the values 8 and 9. Thus, $\inv(T)= 4$ and $\maj(T)=3$. }\label{triple}
\end{figure}

We are now able to define a combinatorial generalization of the \emph{modified Macdonald polynomials},
\begin{equation}\label{Macdonalds}
\widetilde H_\delta(X; q, t) := \sum_{ { T \in \SF(\delta)} } 
q^{\inv(T)}t^{\maj(T)}F_{\sigma(T)}.
\end{equation}

\noindent We similarly define the \emph{modified Hall-Littlewood polynomials} as
\begin{equation}\label{Halls}
\HLdelta := \sum_{ { T \in \SF(\delta)} \atop \inv(T)=0} 
t^{\maj(T)}F_{\sigma(T)}.
\end{equation}

\noindent
Further, if we let $\gamma\subset\delta$, we may define
\begin{equation}
\Des:=\sum_{ { T \in \SF(\delta)} \atop {\inv(T)=0 \atop \Descent(T)=\gamma}} F_{\sigma(T)}.
\end{equation}

\noindent
It follows immediately from these definitions that
\begin{equation}
\HLdelta = \sum_{\gamma\subset\delta} t^{\maj_\delta(\gamma)}\Des.
\end{equation}


 Using the involution in \citep{HHL}, it is possible to write $\Macdelta$ as a sum of Lascoux-Leclerc-Thibon polynomials indexed by tuples of (possibly disconnected) ribbons, which are shown to be Schur positive in \citep{grojnowski2007affine}. Hence, $\Macdelta$ is both symmetric and Schur positive. The only distinction between the partition case in \citep{HHL} and this paper is that $\delta$ is not forced to be a partition, and so there are fewer restrictions on the ribbons in the associated Lascoux-Leclerc-Thibon polynomials. Specifically, the ribbons may be disconnected and may not be in descending order of size. While the coefficients of the Schur expansion of $\HLdelta$ and $\Des$ are given by Theorem~\ref{Hall Yam classification}, finding a combinatorial description of this expansion for $\Macdelta$ remains an open problem. 

We should emphasize that our definition of $\Macdelta$ and $\HLdelta$ are combinatorial generalizations, chosen to agree with definitions in \citep{HHL} and related definitions of LLT polynomials. Hence, they need not agree with any algebraic generalizations of Macdonald polynomials. Specifically, Garsia and Haiman conjectured a generalization of Macdonald polynomials to diagram indices in \citep{Garsia:1995:FPR:216279.216300}, with further results contributed by Jason Bandlow in his Ph.D. dissertation \citep{bandlow2007combinatorics}. Their conjecture would require that $\Macdelta=\widetilde H_{\delta^\prime}(X;t,q)$. While this is the case when $\delta$ is a partition, it fails for the diagram $\{(0,0), (1,1)\}$. Finding a way of recovering this symmetry, perhaps by modifying the maj statistic, is an important open problem.

\begin{remark}\label{first 3 rows}\textup{
In order to use Theorem~\ref{Hall Yam classification} to expand $\HLdelta$ in terms of Schur functions, it is necessary to generate $\{ w \in \Yam_\delta(|\delta|) \colon \inv_\delta(w)=0\}$. We may do this by making a tree: proceeding as mentioned in Part 1 of Remark~\ref{yam tree} by filling $\delta$ in reverse row reading order and checking that there are no inversions, that we still have a Yamanouchi word, and that no pistol is jammed with the addition of each new entry. In the case of partition shapes, one can consider such fillings row by row to show that the bottom three rows must satisfy one of the three cases in Figure~\ref{first three rows}. Specifically, the bottom row must be all 1's, the second row starts with $k$ 2's followed by all 1's, and the third row has $j\leq k$ 3's followed by one of three options. Either the rest of the third row is 1's, or there are $k-j$ 1's followed by all 2's, or the rest may be all 2's if the result is still a Yamanouchi word. It is, in theory, possible to precompute more rows in this fashion at the expense of more complicated rules. 
 \begin{figure}[H]
   \begin{center} 
   \ytableausetup{smalltableaux}
\ytableaushort{3311111,2222111, 1111111}
\hspace{.3in}\ytableaushort{3311222,2222111,1111111}
\hspace{.3in}
\ytableaushort{3322222,2222111,1111111}
 \end{center}
 \caption{The three ways to fill the first three rows of $\mu$ when expanding $\HL$ into Schur functions.}\label{first three rows}
 \end{figure}
 It should be noted that the tree described above may still have dead ends. In that respect, a key open problem is to find an algorithm that avoids any dead ends in order to maximize efficiency. Such an algorithm was provided for the Littlewood-Richardson coefficients in \citep{remmel1984multiplying}, suggesting that it may be possible in this case as well. 
}
\end{remark}

\subsection{Dual Equivalence Graphs}

The key tool used in this paper is the theory of dual equivalence graphs. We quickly lay out the necessary background on the subject in this section.

\begin{dfn}[\cite{Haiman}]\label{def: d_i}\textup{
Given a permutation in $S_n$ expressed in one-line notation, define an \emph{elementary dual equivalence} as an involution $d_i$ that interchanges the values $i-1, i,$ and $i+1$ as
}
\begin{equation}\label{dual equiv}
\begin{split}
d_i(\ldots i \ldots i-1 \ldots i+1 \ldots) = (\ldots i+1 \ldots i-1 \ldots i \ldots),\\
d_i(\ldots i-1 \ldots i+1 \ldots i \ldots) = (\ldots i \ldots i+1 \ldots i-1 \ldots),
\end{split}
\end{equation}
\textup{and that acts as the identity if $i$ occurs between $i-1$ and $i+1$. Two permutations are $dual \; equivalent$ if one may be transformed into the other by successive elementary dual equivalences.
}
\end{dfn}

\noindent For example, 21345 is dual equivalent to 41235 because
$d_3(d_2( 21345 )) = d_3( 31245 ) =  41235$. 

We may also let $d_i$ act on the entries of a standard Young tableau by applying them to the row reading word. It is not hard to check that the result of applying this action to a standard Young tableau is again a standard Young tableau. The equivalence classes of Young tableaux under dual equivalence are described in the following theorem.

\begin{thm}[{\cite[Prop. 2.4]{Haiman}}]\label{Haiman}
Two standard Young tableaux on partition shapes are dual equivalent if and only if they have the same partition shape.
\end{thm}

\noindent
It follows from \cite[Lem. 2.3]{Haiman} that the action of $d_i$ on $S_n$ is further related to its action on $\SYT(n)$ by
\begin{equation}\label{d and P}
d_i(P(\pi))=P(d_i(\pi)).
\end{equation}

By definition, $d_i$ is an involution, and so we can define a graph on standard Young tableaux by letting each nontrivial orbit of $d_i$ define an edge colored by $i$. By Theorem \ref{Haiman}, the graph on SYT($n$) with edges colored by $1 < i < n$ has connected components with vertices in SYT($\lambda$) for each $\lambda \vdash n$. We may further label each vertex with its signature to create a \emph{standard dual equivalence graph} that we will denote $\mathcal{G}_\lambda=(V, \sigma, E)$, where $V=\SYT(\lambda)$ is the vertex set, $\sigma$ is the signature function, and $E$ is the colored edge set. Refer to Figure~\ref{5graph} for examples of $\mg_\lambda$ with $\lambda\vdash 5$.
%
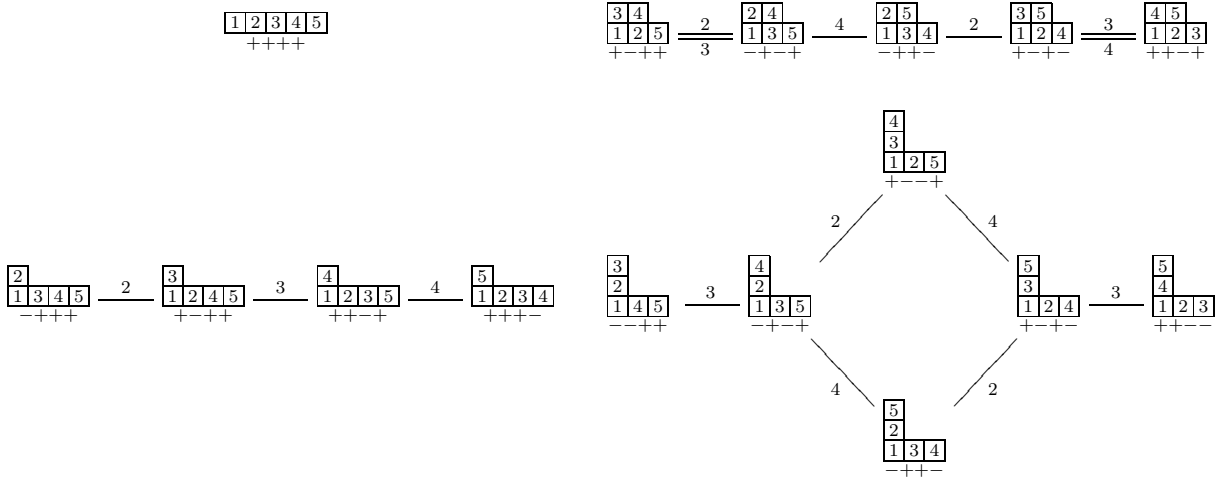
\begin{figure}[h]
\[ \begin{array}{cc}
  \vcenter{\scalebox{.85}{ \vbox{
\xymatrix{ & && {\;\; \Yvcentermath1 \scriptsize {\young(12345) \atop ++++}}}
}}}
& \hspace{-3.7in}
\vcenter{\scalebox{.85}{ \vbox{
	\xymatrix{
		{\Yvcentermath1 \scriptsize {\young(34,125) \atop +-++}} \ar@{=}[r]^{2}_{3} & {\Yvcentermath1 \scriptsize {\young(24,135)}\atop -+-+} \ar@{-}[r]^{4}	& {\Yvcentermath1 \scriptsize {\young(25,134)}\atop -++-} \ar@{-}[r]^2	& { \Yvcentermath1 \scriptsize {\young(35,124)}\atop +-+-} \ar@{=}[r]^3_4 & {\Yvcentermath1 \scriptsize {\young(45,123) \atop ++-+}} 
		} }}}
 \\ & \\
\vcenter{\scalebox{.85}{ \vbox{ \xymatrix{
		{\Yvcentermath1 \scriptsize {\young(2,1345) \atop -+++}} \ar@{-}[r]^{2} 	& {\Yvcentermath1 \scriptsize {\young(3,1245)}\atop +-++} \ar@{-}[r]^{3}	& {\Yvcentermath1 \scriptsize {\young(4,1235)}\atop ++-+} \ar@{-}[r]^4	& { \Yvcentermath1 \scriptsize {{\young(5,1234) \atop +++-}}
} }} } }
&\hspace{-3.7in}
	\vcenter{\scalebox{.85}{ \vbox{ \xymatrix{
		&&{\Yvcentermath1 \scriptsize {\young(4,3,125)}\atop +--+} \ar@{-}[rd]^{4} &&\\
		{\Yvcentermath1 \scriptsize {\young(3,2,145)}\atop --++}\; \ar@{-}[r]^{3}	& {\Yvcentermath1 \scriptsize{\young(4,2,135)}\atop -+-+} \ar@{-}[ur]^{2} \ar@{-}[rd]_{4} && {\Yvcentermath1 \scriptsize{ \young(5,3,124)}\atop +-+-} \ar@{-}[r]^{3} & {\Yvcentermath1 \scriptsize{\young(5,4,123)}\atop ++--}\\
		&&{\Yvcentermath1 \scriptsize {\young(5,2,134)}\atop -++-} \ar@{-}[ur]_{2} &&
		} }} } 
\end{array} \]
\caption{The standard dual equivalence graphs on all $\lambda \vdash 5$ up to conjugation.}
 \label{5graph}
\end{figure}


Definition \ref{Ges} and Theorem \ref{Haiman} determine the connection between Schur functions and dual equivalence graphs as highlighted in \cite{assaf2015dual}. Given any standard dual equivalence graph $\mg_\lambda=(V,\sigma, E)$,
\begin{equation} \label{schurs}
\sum_{v\in V} F_{\sigma(v)} = s_\lambda.
\end{equation}

\noindent Here, $\mathcal{G}_\lambda$ is an example of the following broader class of graphs.

\begin{dfn}
\textup{
A \emph{signed colored graph} consists of the following data:}
\begin{enumerate}
\vspace{-.02in}
\item a finite vertex set $V$,
\item a signature function $\sigma:V \rightarrow \{\pm1\}^{n-1}$ for some fixed positive integer $n$,
\item a collection $E_i$ of unordered pairs of distinct vertices in $V$ for each  $i \in \{2, \ldots, n-1\}$ and the same positive integer $n$.
\end{enumerate}

\noindent \textup{We denote a signed colored graph by $\mg =(V, \sigma, E_{2}\cup \cdots \cup E_{n-1})$ or simply $\mg=(V, \sigma, E)$.
}
\end{dfn}

In order to give an abstract definition of dual equivalence graphs, we will need definitions for isomorphisms and restrictions.
Given two signed colored graphs $\mg(V,\sigma,E)$ and $\mh(V^\prime, \sigma^\prime, E^\prime)$, an \emph{isomorphism} $\phi\colon \mg\to \mh$ is a bijective map from $V$ to $V^\prime$ such that both $\phi$ and $\phi^{-1}$ preserve colored edges and signatures. The definition of a restriction is a bit more technical.

\begin{dfn}\label{I restriction}
\textup{Given a signed colored graph $\mg=(V, \sigma, E)$ and interval of positive integers $I$, define the restriction of $\mg$ to $I$, denoted $\mg|_I$, as the signed colored graph $\mh=(V, \sigma^\prime, E^\prime)$, where
\begin{enumerate}
\item 
$\sigma^\prime(v)_i=\sigma(v)_{\min(I)+i-1}$ 
when $i \in \{1, \ldots, |I|-1\}$ and $\sigma_{\min(I)+i-1}$ is defined.
\item
$E_i^\prime= E_{\min(I)+i-1}$
when $i \in \{2, 3, \ldots, |I|-1\}$ and $E_{\min(I)+i-1}$ is defined.
\end{enumerate}
}
\end{dfn}

One useful example of Definition~\ref{I restriction} arises by considering the restriction of $\mg_\lambda$ to an integer interval $I\subset [|\lambda|]$, as in Figure~\ref{restricting to skew}. In this case, we can modify the vertex set  of restricted graph by removing cells with values not in $I$ and then lowering the remaining values by $\min(I)-1$. Edges and signatures are then given by the action of $d_i$ and $\sigma$, respectively, on these new vertices.

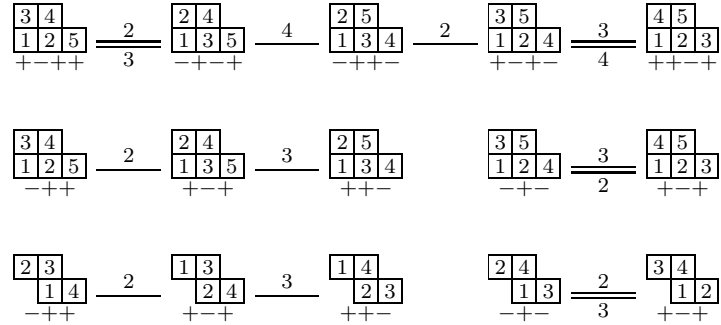
\begin{figure}[h]

\[
\scalebox{1}{ \vbox{
	\xymatrix{
		{\Yvcentermath1 \scriptsize {\young(34,125) \atop +-++}} \ar@{=}[r]^{2}_{3} & {\Yvcentermath1 \scriptsize {\young(24,135)}\atop -+-+} \ar@{-}[r]^{4}	& {\Yvcentermath1 \scriptsize {\young(25,134)}\atop -++-} \ar@{-}[r]^2	& { \Yvcentermath1 \scriptsize {\young(35,124)}\atop +-+-} \ar@{=}[r]^3_4 & {\Yvcentermath1 \scriptsize {\young(45,123) \atop ++-+}} 
		} }}
\]\vspace{.1in}\[
\scalebox{1}{ \vbox{
	\xymatrix{
		{\Yvcentermath1 \scriptsize {\young(34,125) \atop -++}} \ar@{-}[r]^{2} & {\Yvcentermath1 \scriptsize {\young(24,135)}\atop +-+} \ar@{-}[r]^{3}	& {\Yvcentermath1 \scriptsize {\young(25,134)}\atop ++-} 	& { \Yvcentermath1 \scriptsize {\young(35,124)}\atop -+-} \ar@{=}[r]^3 _2& {\Yvcentermath1 \scriptsize {\young(45,123) \atop +-+}} 
		} }}
\]\vspace{.1in}\[
\scalebox{1}{ \vbox{
	\xymatrix{
		{\Yvcentermath1 \scriptsize {\young(23,:14) \atop -++}} \ar@{-}[r]^{2} & {\Yvcentermath1 \scriptsize {\young(13,:24)}\atop +-+} \ar@{-}[r]^{3}	& {\Yvcentermath1 \scriptsize {\young(14,:23)}\atop ++-} 	& { \Yvcentermath1 \scriptsize {\young(24,:13)}\atop -+-} \ar@{=}[r]^2_3 & {\Yvcentermath1 \scriptsize {\young(34,:12) \atop +-+}} 
		} }}		
\]
\caption{The standard dual equivalence graph $\mg_{(3,2)}$ above its restriction to $I=\{2, 3, 4, 5\}$ in the center row. In the bottom row, each vertex relabeled by the result of subtracting 1 from each entry of $T|_I$ so that edges and signatures are given by the actions of $d_i$ and $\sigma$, respectively. }\label{restricting to skew}
\end{figure}

We now proceed to the definition of a dual equivalence graph. Here, we use results in \citep{DEG+LLT} as our definition (see Remark~\ref{DEG def remark} for details). For  more general definitions, see \citep{assaf2015dual} and \citep{DEG+LLT}.

\begin{dfn} \label{properties}  
\textup{A signed colored graph $\mathcal{G} = (V, \sigma, E)$ is a \emph{dual equivalence graph}  if the following two properties hold. 
\begin{enumerate}[align=left, leftmargin=*]
\item[\emph{Locally Standard Property}:]  If $I$ is any interval of integers with $|I|=6$, then each component of $\mg|_I$ is isomorphic to some $\mg_\lambda$.
\item[\emph{Commuting Property}:]  If $\{v,w\} \in E_i$ and $\{w,x\}  \in E_j$ for some $|i-j| > 2$, then there exists $y\in V$ such that $\{v,y\} \in E_j$ and $\{x,y\} \in E_i$.
\end{enumerate} }
\end{dfn}

\begin{remark}\label{DEG def remark}
\textup{
The definition of dual equivalence graph in \citep{assaf2015dual} requires that a signed colored graph satisfy six axioms. Theorem~3.17 of \citep{DEG+LLT} eliminates Axiom~6 by strengthening Axiom 4, replacing it with Axiom $4^+$. The Commuting Property is a restatement of Axiom~5. As mentioned in Part~2 of Remark~3.18 in \citep{DEG+LLT}, Axioms 1, 2, 3, and $4^+$ are equivalent to our Locally Standard Property. More specifically, Axiom $4^+$ dictates the possible unsigned graph structures of 4 consecutive colors, while Axioms 1-3 dictate the relationships between $i$-edges and the signatures of their vertices. In fact, given these axioms, the edges fully determine the signatures up to multiplying every signature by $-1$, as is also the case with our Locally Standard Property. The only logical distinction between the result in  \citep{DEG+LLT} and the definition given here is that the former, for the sake of their inductive argument, considers graphs with some sets of $E_i$ omitted.
}
\end{remark}




\begin{thm}[{\citep[Thm~3.7]{assaf2015dual}}]\label{main}\label{Tableaux are DEGs}
A connected component of a signed colored graph is a dual equivalence graph if and only if it is isomorphic to a unique $\mg_\lambda$. 
\end{thm}

Next, we will associate with every Macdonald polynomial and Hall-Littlewood polynomial a signed colored graph. To do this, we need to define an involution $D_i^\delta$ to provide the edge sets of a signed colored graph, as defined in \citep{assaf2015dual}. First let 
 $\tilde d_i:S_n \rightarrow S_n$ be the involution that cyclically permutes the values $i-1, i,$ and $i+1$ as
\begin{equation}
\begin{split}
\tilde d_i(\ldots i \ldots i-1 \ldots i+1 \ldots) = (\ldots i-1 \ldots i+1 \ldots i \ldots), \\
\tilde d_i(\ldots i \ldots i+1 \ldots i-1 \ldots) = (\ldots i+1 \ldots i-1 \ldots i \ldots),
\end{split}
\end{equation}
\noindent and that acts as the identity if $i$ occurs between $i-1$ and $i+1$. For example, $\tilde d_3 \circ \tilde d_2(4123) = \tilde d_3(4123) = 3142.$

We now define the desired involution. Given $\pi \in S_n$ and a diagram $\delta$ of size $n$,

\begin{equation} \label{def: D}
D_i^{\delta}(\pi) :=
 \begin{cases}
  \tilde d_i(\pi) & \textup{if $i-1$, $i$, and $i+1$ are in a pistol of  $T_\delta(\pi)$}\\
  d_i(\pi) & \textup{otherwise.}
\end{cases}
\end{equation}

\noindent As an example, we may take $\pi=53482617$ and $\delta$ as in Figure \ref{tau2}. Then $D_3^{\delta}(\pi)=\tilde d_3(\pi)=54283617$ and $D_5^{\delta}(\pi)=d_5(\pi)=63482517$. 

\begin{figure}[H]
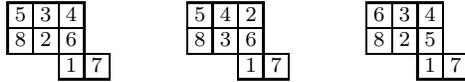

  \ytableausetup{smalltableaux,aligntableaux=bottom}
  \[ 
   \ytableaushort{534, 826, \none\none 17} \hspace{.4in} 
   \ytableaushort{542, 836, \none\none 17} \hspace{.4in} 
   \ytableaushort{634, 825, \none\none 17} \vspace{-.1in}
  \]
  \caption{Three standard fillings of a diagram $\delta$. At left, a standard filling with row word $\pi=53482617$ followed by $T_\delta(D_3^{\delta}(\pi))$ and then $T_\delta(D_5^{\delta}(\pi))$.}
\label{tau2}
\end{figure}
Given some $\delta$ of size $n$, we may then define an \emph{Assaf Graph} as the signed colored graph $\mh_{\delta}=(V, \sigma, E)$ with vertex set $V=S_n$, signature function $\sigma$ given via inverse descents, and edge sets $E_i$ defined via the nontrivial orbits of $D^{\delta}_i$. 
 It is readily shown that the action of $D_i^{\delta}$ on $\pi$ preserves $\inv_\delta(\pi)$, $\Descent_\delta(\pi)$, and  $\maj_\delta(\pi)$ (see \citep{assaf2015dual}). Thus, these functions are all constant on components of $\mh_{\delta}$. 
We may study $\HLdelta$ by restricting our attention to components of $\mh_{\delta}$ where $\inv_\delta$ is zero, as in the following definition.

\begin{dfn}\label{HL graph}
\textup{
Let $\gamma$ and $\delta$ be diagrams such that $\gamma \subset \delta$ and $|\delta|=n$. We define two subgraphs of $\mh_\delta=(V,\sigma,E)$ induced by restricting the vertex set $V$ as follows.
\begin{enumerate}
\item $\mP_\delta :=(V^{\prime}, \sigma, E^{\prime})$ is the subgraph with $V^{\prime}=\{ \pi\in S_{n} \colon \inv_\delta(\pi)=0\}$.
\item $\mr_{\gamma, \delta} :=(V^{\prime\prime}, \sigma, E^{\prime\prime})$ is the subgraph with 
$V^{\prime\prime}=\{ \pi\in S_{n} \colon \inv_\delta(\pi)=0, \Descent_\delta(\pi)=\gamma \}$.
%
%
\end{enumerate}}
\end{dfn}

Notice that each subgraph is a union of connected components of $\mh_\delta$. 
 Furthermore, 
\begin{align}
\begin{split}
\Macdelta= &\sum_{v\in V} q^{\inv_\delta(v)}t^{\maj_\delta(v)}F_{\sigma(v)}\\ \HLdelta=& \sum_{v\in V^{\prime}} t^{\maj_\delta(v)}F_{\sigma(v)},\\ \label{P by graph}
\Des=& \sum_{v\in V^{\prime\prime}} F_{\sigma(v)}.
\end{split}
\end{align}

The Assaf graph $\mh_\delta$ is the primary object of interest in the proof of the following lemma, which was originally stated in terms of Lascoux-Leclerc-Thibon polynomials but is easily translated using results in \citep{HHL}.

\begin{lem}[{\cite[Thm.~4.3]{DEG+LLT}}]\label{3 cell pistols}
Let $\delta$ be a diagram such that no pistol of $\delta$ contains more than three cells. Then
\[
\Macdelta = \sum_{\lambda \vdash |\delta|} \;
\sum_{w \in \Yam(\lambda)}  q^{\inv_\delta(w)}t^{\maj_\delta(w)}s_\lambda.
\]
\end{lem}
\begin{proof}
For the sake of completeness, we briefly sketch how to translate \citep[Thm.~4.3]{DEG+LLT} into the form used in Lemma~\ref{3 cell pistols}. Every standard filling of $\delta$ can be taken to a unique standard filling of a tuple of not-necessarily connected ribbon tableaux $\bnu$, as described in \citep{HHL}, taking the row reading word of a filling of $\delta$ to the content reading word of $\bnu$. In this way, $\Macdelta$ can be expanded into a sum of LLT polynomials, where each term is multiplied by an appropriate power of $q$ and $t$. The requirement that each pistol of $\delta$ contains no more than three cells is equivalent to requiring that the diameter of $\bnu$ is at most 3, as defined in \citep{DEG+LLT}. If the diameter of $\bnu$ is at most 3, then  \citep[Thm.~4.3]{DEG+LLT} states that the Schur expansion of the related LLT polynomial may be retrieved as a sum over fillings of $\bnu$ whose reading words are Yamanouchi words. Lemma~\ref{3 cell pistols} is then an immediate result of sending fillings of $\delta$ to fillings of $\bnu$ and applying this Schur expansion.  
\end{proof}


\section{Dual Equivalence graphs in $\mh_{\delta}$}\label{DEGs of tau}

\subsection{The graphs $\mP_\delta$ and $\mr_{\gamma,\delta}$}

Next we give a classification for when a component of $\mh_{\delta}$ is a dual equivalence graph. To do this, we will need the following definition. 
%
Given permutations $p\in S_m$ and $\pi \in S_n$ with $m \leq n$, we say that $p$ is a \emph{strict pattern} of $\pi$ if there exists some sequence $i_1 < i_2 < \ldots < i_m$ such that $\pi_{i_j}=p_j+k$ for some fixed integer $k$ and all $1\leq j\leq m$. Furthermore, we say that $p$ is a $\delta$-strict pattern of $\pi$ if the indices $i_1, i_2, \ldots, i_m$ of $\pi$ are $\delta$-pistoled. 
 As an example, on the left side of Figure~\ref{tau2},  $p=231$ is a strict pattern of $\pi=53482617$ occurring in $\pi$ at indices 2, 3, and 5. These indices correspond to the second, third, and fifth cell in row reading order of the given diagram $\delta$. These cells are contained in a pistol of $\delta$.  Thus, $p=231$ is a $\delta$-strict pattern of $\pi$. 

\begin{lem}\label{DEG by pattern}
Let $\mg=(V,\sigma, E)$ be a component of $\mh_{\delta}$. Then $\mg$ is a dual equivalence graph if and only if  for every vertex $\pi\in V$, the following two conditions hold.
\begin{enumerate}
\item The permutations 1342 
 and 2431 are not $\delta$-strict patterns of $\pi$.
\item If the permutation 12543 or 34521 is a strict pattern of $\pi$ occurring in indices $i_1, \ldots, i_5$, then $\{i_1, \ldots, i_5\}$ is $\delta$-pistoled, $\{i_1, \ldots, i_4\}$ is not $\delta$-pistoled, or $\{i_2, \ldots, i_5\}$ is not $\delta$-pistoled.
\end{enumerate}
\end{lem}
\begin{proof}
First assume that $\mg$ contains a vertex with one of the patterns described above. We will prove that $\mg$ is not a dual equivalence graph by showing that it does not satisfy the Locally Standard Property in Definition~\ref{properties}. Consider the smallest interval $I$ such that restricting to $\pi|_I$ gives a word that still contains one of the strict patterns. Here, $|I|=4 \textup{ or } 5$ depending on whether the pattern is in the first or second case above, respectively. We may then consider the component of $\pi$ in $\mg|_I$. There are four possibilities corresponding to the four strict patterns, 1342, 2431, 12543, and 3452. Direct inspection shows that each such component does not satisfy the Locally Standard Property, as demonstrated in Figure~\ref{bad graphs}.

\begin{figure}[h]
\[
\xymatrix{	
& 2314 \ar@{-}[r]^{2} & 3124  \ar@{-}[r]^{3} &2143 \ar@{-}[r]^{2} &1342 \ar@{-}[r]^{3} & 1423  & \\
& 4132 \ar@{-}[r]^{2} & 4213  \ar@{-}[r]^{3} &3412 \ar@{-}[r]^{2} &2431 \ar@{-}[r]^{3} & 3241  & \\
 21453\ar@{-}[r]^{4} \ar@{=}[d]_{2}^{3} & 21534 \ar@{-}[r]^{2} & 13524 \ar@{-}[d]^{3}\ar@{-}[r]^{4} & 14325 \ar@{-}[r]^{2} & 24135 \ar@{-}[r]^{4}\ar@{-}[d]^{3} & 23154 \ar@{-}[r]^{2} &31254 \ar@{=}[d]_{3}^{4}  \\ 
31452 &&12543 & & 32145 &&  41253
\\
 35412\ar@{-}[r]^{4} \ar@{=}[d]_{2}^{3} & 43512 \ar@{-}[r]^{2} & 42531 \ar@{-}[d]^{3}\ar@{-}[r]^{4} & 52341 \ar@{-}[r]^{2} & 53142 \ar@{-}[r]^{4}\ar@{-}[d]^{3} & 45132 \ar@{-}[r]^{2} &45213 \ar@{=}[d]_{3}^{4}  \\ 
25413 &&34521 & & 54123 && 35214
}
\]
\caption{From top to bottom, the components of 1342, 2431, 12543, and 34521 with edges given by $D_i$ and pistols as described in Lemma~\ref{DEG by pattern}. Signatures are omitted for clarity.}\label{bad graphs}
\end{figure}
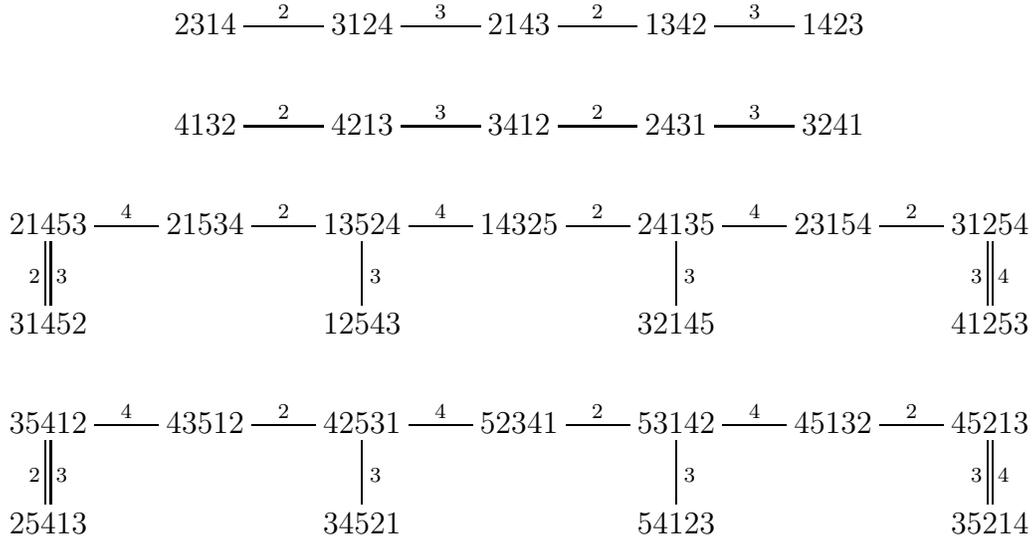

To prove the other direction, assume that $\mg$ does not contain a vertex with one of the above patterns. We need only show that $\mg$ satisfies Definition~\ref{properties}. The proof of the Commuting Property follows from the fact that $D_i^\delta$ fixes all values except $i-1, i,$ and $i+1$. To demonstrate the Locally Standard Property, it suffices to check $\mh_{\delta}$ with $|\delta| = 6$. While it is possible to meticulously do this check by hand, it is more straightforward to verify the finitely many cases by computer. More specifically, it suffices to consider the possible graphs on $S_6$. Here, we must consider the 132 different ways of grouping the six indices into pistols separately, since the action of $D_i$ is dependent on the choice of these pistols. This finite check completes the proof.
\end{proof}

\begin{thm}\label{Halls are DEGs}
If $\gamma$ and $\delta$ are any diagrams such that $\gamma\subset\delta$, then $\mr_{\gamma,\delta}$ and $\mP_{\delta}$ are dual equivalence graphs.
\end{thm}
\begin{proof}
 Let $\pi$ be some arbitrary vertex of $\mr_{\gamma,\delta}$ or $\mP_\delta$. 
 It suffices to show that if the hypotheses of Lemma~\ref{DEG by pattern} are not satisfied by a permutation $\pi$ and diagram $\delta$, then $\inv_\delta(\pi)\neq 0$. Let $T=T_\delta(\pi)$ such that $\pi$ contains one of the strict patterns mentioned in Lemma~\ref{DEG by pattern}. Let $\pi|_I$ be this pattern. We will use the location of $\pi|_I$ in $T$ to show that $\inv_\delta(\pi)\geq 1$.

First notice that as a word, $\pi|_I$ ends in a descent. If the last two values of $\pi|_I$ occur in the same row of $T$, they must be part of an inversion triple or an inversion pair, since the cell completing said triple cannot have a value in between the last two values in $\pi|_I$, by the definition of a strict pattern. We may thus restrict our attention to the case where the last value of $\pi|_I$ does not share its row in $T$ with any other value of $I$.

We will now demonstrate an inversion triple or inversion pair by ignoring any rows and columns that do not contain the last four values of $\pi|_I$. By the assumption that the last four values of $\pi|_I$ are contained in a pistol, we have restricted to a diagram with exactly two rows. We may then demonstrate an inversion triple or an inversion pair in all possible cases, as shown in Figure~\ref{Hall Forbidden}. 
Thus $\inv_\delta(\pi)\geq 1$, completing our proof.
\end{proof}
\begin{figure}[H]
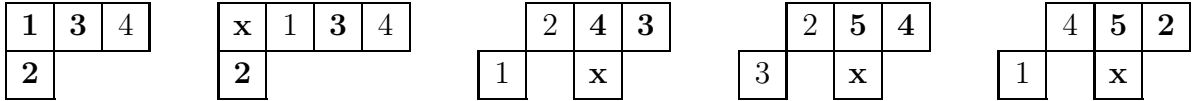
 
   \begin{center} 
   \ytableausetup {boxsize=normal,mathmode}
\ytableaushort{{\bf1}{\bf3}4,{\bf2}\none \none}
\hspace{.3in}
\ytableaushort{{\bf x}1{\bf3}4,{\bf2}\none\none\none}
\hspace{.3in}
\ytableaushort{\none2{\bf4}{\bf3},1\none{\bf x}\none}
\hspace{.3in}
\ytableaushort{\none2{\bf5}{\bf4},3\none{\bf x}\none}
\hspace{.3in}
\ytableaushort{\none4{\bf5}{\bf2},1\none{\bf x}\none}
   \end{center}
  \caption{Configurations for the last four values of $\pi|_I$ with inversion triples in bold. Here, numbers may be shifted by some nonnegative integer $k$, and the variable $x$ may represent an omitted cell or any value not in $I$.}
 \label{Hall Forbidden}
\end{figure}
%

%
%
%

\subsection{The proofs of Theorems~\ref{Hall Yam classification} and \ref{yam decomposition of graphs}}
We begin by giving several lemmas necessary for the proof of  Theorem~\ref{yam decomposition of graphs}. These lemmas, however, may be safely skipped without hindering the understanding of later results. In the proof of Lemmas~\ref{move n-1}--\ref{yam to yam}, we will use the variables $s$, $t$, $u$, and $v$ to denote vertices of graphs. Moreover, the vertices of graphs in the proofs of these lemmas will always be permutations.

\begin{lem}\label{move n-1}
Let $\gamma$ be a diagram such that $|\gamma|=n$, and let $\lambda\vdash n$ have at most two rows. Let $\pi\in\SYam_\gamma(\lambda)$ and  $u=\pi|_{[n-1]}$. Let $\delta$ be the diagram of $T_\gamma(\pi)\big|_{[n-1]}$. Then $u$ is connected by a path $p$ in $\mh_\delta$ to some vertex $v$ such that:
\begin{enumerate}
\item each edge of $p$ is defined via the action of $d_i$, 
\item$v_{n-1}=n-1$,
\item $u_i=v_i$ whenever $u_i<\lambda_1$. 
\end{enumerate}
Furthermore, given $\gamma$ and $\lambda$, there is a sequence of edge colors for $p$ that is not dependent on the choice of $\pi \in \SYam_\gamma(\lambda)$.
\end{lem}
\begin{proof}
Before beginning, it is worth giving a further description of $u$ and $v$ when $\lambda_2>1$. An example of such a $u$,  $v$, and $p$ is in Figure~\ref{specific sequence}. Because $u$ is the result of removing $n$ from a standardized Yamanouchi word, $\unst(u)$ will always have strictly more $1$'s than $2$'s when read one value at a time from right to left. Meanwhile, $\unst(v)$ is the result of replacing $\unst(u)_{n-1}$ with a 2. Row inserting $u$ gives the two row tableau with values of $[1, \lambda_1]$ in the first row, while row inserting $v$ gives the two row tableau with values of $[1, \lambda_1-1]\cup \{n-1\}$ in the first row. Recalling that the action of $d_i$ on $P(u)$ can be understood via (\ref{d and P}), the path $p$ to be described in this proof may be portrayed via insertion tableaux, as in Figure~\ref{move n-1 fig}.

We proceed by induction on $\lambda_2$. If $\lambda_2\leq 1$, then $u$ must be the identity permutation $123\cdots n-1$, since $\pi \in \SYam(\lambda)$. We may then let $u=v$ in order to satisfy the result. For the inductive step, suppose that $\lambda_2>1$ and that the result holds for all two row partitions whose second row is smaller than $\lambda_2$.  

We will apply a sequence of edges to find $v$. It follows from the fact that $\pi \in \SYam(\lambda)$ and $\lambda_2 \geq 2$ that $u_{n-2}=\lambda_1-1$ and $u_{n-1}=\lambda_1$. We may apply the inductive hypotheses by replacing $\gamma$ and $\pi$ with $\delta$ and $u$ in the statement of the lemma.  
By induction, move $n-2$ into position $n-1$ of $u$. Call this vertex $u'$. 

Similarly restricting to values in $[n-3]$, we may then move $n-3$ into position $n-2$ via some path $q$. 
This application of Lemma~\ref{move n-1} is more intricate. Specifically, we may let $\pi$ be the result of omitting $n-2$ in $u'$ and then replacing $n-1$ with $n-2$. That is, $\unst(\pi)$ is the result of omitting the furthest right 1 and 2 from $\unst(u)$, and the underlying diagram $\gamma'$ is the result of omitting the two cells of $\gamma$ containing these values.  Thus $\pi\in\SYam_{\gamma'}(\lambda_1-1, \lambda_2-1)$. 

We have not changed the index of $n-1$, and so $n-1$ must now occur before $n-3$, which occurs before the $n-2$ in the last index. Applying an $n-2$-edge thus moves $n-1$ into the last index. Since the last index of $\pi$ does not jam a pistol of $\gamma$, 
 the last index of $u$ and the index of $n-1$ in $u$ cannot be $\delta$-pistoled. In particular, the $n-2$-edge must be defined via $d_i$. Finally, we may consider the restriction to values in $[n-3]$ again to apply the edges of $q$ in reverse order, ensuring $u_i=v_i$ whenever $u_i<\lambda_1$. The result is the desired $v$, as given by applying a sequence of edges that was not dependent on the choice of $\pi\in\SYam(\lambda)$.
\end{proof}
%
\begin{figure}[h]
\ytableausetup{aligntableaux=center,smalltableaux=on}
\begin{center}$
\xymatrix{
T_\gamma(\pi)= \ytableaushort{51,67,234} & T_\delta(u)=\hspace{-.35in}& \ytableaushort{51,6,234} \ar@{=}[r]^4_5&  \ytableaushort{41,6,235} \ar@{-}[r]^3&  \ytableaushort{31,6,245} \ar@{-}[r]^5&  \ytableaushort{31,5,246}\ar@{=}[r]^3_4& \ytableaushort{41,5,236}& \hspace{-.35in}=T_\delta(v)
}$
\end{center}
 \ytableausetup{aligntableaux=center,boxsize=.3in}
\caption{At left, an example of $\pi\in \SYam_\gamma((4, 3))$ from Lemma~\ref{move n-1}, viewed as a filling of $\gamma=(3, 2, 2)$. At right, a path connecting $u$ and $v$, viewed as fillings of $\delta$. The first 3-edge is $q$, and the second is $q^{-1}$.}
\label{specific sequence}
\end{figure}
%
%
\begin{figure}[h]
\ytableausetup{aligntableaux=center,boxsize=.3in}
\scalebox{.8}{
\xymatrix{	
	P(u)=\hspace{-.3in}&\ytableaushort{{\scriptstyle \lambda_1+1}{ \scriptstyle \lambda_1+2}{}{\scriptstyle n-2}{\scriptstyle n-1},12{}{}{\scriptstyle \lambda_1-1}{\scriptstyle \lambda_1}}
	 \,\ar@{--}[r]&
	 \ytableaushort{{\scriptstyle \lambda_1}{ \scriptstyle \lambda_1+1}{}{\scriptstyle n-3}{\scriptstyle n-1},12{}{}{\scriptstyle \lambda_1-1}{\scriptstyle n-2}} 
	 \,\ar@{--}[r]^{q}&
	 \ytableaushort{{\scriptstyle \lambda_1-1}{ \scriptstyle \lambda_1}{}{\scriptstyle n-4}{\scriptstyle n-1},12{}{}{\scriptstyle n-3}{\scriptstyle n-2}}
&\\ &
	 \ytableaushort{{\scriptstyle \lambda_1-1}{ \scriptstyle \lambda_1}{}{\scriptstyle n-4}{\scriptstyle n-1},12{}{}{\scriptstyle n-3}{\scriptstyle n-2}}
	 \,\ar@{-}[r]^{n-2} &
	 \ytableaushort{{\scriptstyle \lambda_1-1}{ \scriptstyle \lambda_1}{}{\scriptstyle n-4}{\scriptstyle n-2},12{}{}{\scriptstyle n-3}{\scriptstyle n-1}}
	  \,\ar@{--}[r]^{q^{-1}} &
	  \ytableaushort{{\scriptstyle \lambda_1}{ \scriptstyle \lambda_1+1}{}{\scriptstyle n-3}{\scriptstyle n-2},12{}{}{\scriptstyle \lambda_1-1}{\scriptstyle n-1}} & \hspace{-.3in}=P(v)
	 }} 
\caption{The sequence of edges connecting $P(u)$ and $P(v)$.}
\label{move n-1 fig}
\end{figure}

\begin{lem}\label{yam to yam induction}
Given diagrams $\gamma$ and $\delta$ such that $|\gamma|=|\delta|=n$, let $\mg$ and $\mh$ be connected components of $\mh_\gamma$ and $\mh_{\delta}$, respectively.  Suppose that there exists an isomorphism $\phi \colon \mg\to \mh$.  Let $\lambda\vdash n$ have exactly two rows, and let $s$ be a vertex of $\mg$ such that $s\in \SYam_\gamma(\lambda)$ and $\phi(s)|_{[n-1]}\in \SYam_{\delta'}(\lambda_1, \lambda_2-1)$, where $\delta'$ is the shape of $T_\delta(\phi(s))|_{[n-1]}$. Then $\phi(s)\in \SYam_\delta(\lambda)$.
\end{lem}
\begin{proof}
Let $s$ be as in the statement of the lemma. We proceed by induction on the size of $\lambda_2$. If $\lambda_2 = 1$, then it follows from the hypothesis that $s\in\SYam(\lambda)$ and the fact that the isomorphism $\phi$ preserves signatures that $\unst(\phi(s))\in\Yam(\lambda)$ and has exactly one 2. Hence, $\phi(s)$ cannot jam $\delta$, by Definition~\ref{jammed}, and $\phi(s)\in\SYam_\delta(\lambda)$.

Next assume that $\lambda_2\geq 2$. Applying induction, we further assume the result for all two row partitions whose second row has fewer than $\lambda_2$ cells. We break the proof into two parts: proving that $\phi(s)\in\SYam(\lambda)$ and then that $\phi(s)$ does not jam $\delta$.
\\

\noindent
\emph{Proof that $\phi(s)\in\SYam(\lambda)$}:
By hypothesis, we know that $\phi(s)|_{[n-1]}\in\SYam(\lambda_1,\lambda_2-1)$, but we need to show that $\phi(s)\in\SYam(\lambda)$. It thus suffices to show that $n$ is in the second row of $P(\phi(s))$. We do this by first finding a particular vertex $u$ in $\mg$ such that $n$ is known to be in the second row of $P(\phi(u))$, as shown via example in Figure~\ref{big specific sequence}. All of the edges described in the proof are defined via the action of $d_i$, so (\ref{d and P}) allows us to view them via insertion tableaux in Figure~\ref{sequence of edges}.

By considering the restriction to values in $[n-2]$, we may apply Lemma~\ref{move n-1} to move $n-2$ into the index of $\lambda_1$ in $s$, reaching some vertex $t$. Specifically,  $s|_{[n-1]}\in \SYam_{\gamma'}((\lambda_1, \lambda_2-1))$, where $\gamma'$ is the diagram of $T_\gamma(s)|_{[n-1]}$, so we may let $s|_{[n-1]}$ be $\pi$ in the statement of Lemma~\ref{move n-1}.

 Next, consider the restriction of $t$ to values in $[n-3]$ to similarly move $n-3$ into the index of $\lambda_1-1$, reaching some vertex $u$. This application of Lemma~\ref{move n-1} is more involved. Specifically, we may let $\pi$ be the result of omitting $n-2$ and $n$ in $t$ and then replacing $n-1$ with $n-2$. That is, $\unst(\pi)$ is the result of omitting the furthest right 1 and 2 from $\unst(s)$, and the underlying diagram $\gamma'$ is the result of omitting the two cells of $\gamma$ containing these values.  Thus $\pi\in\SYam_{\gamma'}(\lambda_1-1, \lambda_2-1)$. 
 
Next, we show that $u$ is contained in an $n-2/n-1$-double edge.  Notice that $s$ is connected to $u$ by a sequence of edges whose labels are less than $n-2$. In particular, $n-1$ and $n$ are not moved. Thus, the indices of $n-2$ and $n-1$ in $u$ are not $\gamma$-pistoled, since the index of $\lambda_1$ in $s$ does not jam $\gamma$. Furthermore, $n$ and $n-3$ are between $n-2$ and $n-1$ in $u$, so $u$ must be contained in an $n-2/n-1$-double edge. 

Now consider the effect of the same sequence of edges on $\phi(s)$. In connecting $s$ to $u$, we only applied Lemma~\ref{move n-1} to $s|_{[n-1]}$ and $t|_{[n-2]}$. By the hypotheses of this lemma, $\phi(s)|_{[n-1]}$ also satisfies the hypotheses of $\pi$ in the statement of Lemma~\ref{move n-1}. Thus, $\phi(s)$ is connected to $\phi(u)$ by a sequence of edges that are defined via $d_i$, move $n-2$ to the index of $\lambda_1$, and do not move $n-1$ or $n$. Because isomorphisms preserve edges, $\phi(u)$ must also admit an $n-2/n-1$-double edge. By hypothesis, (\ref{d and P}), $P(s)$ and $P(u)$ have the same shape and have $n$ in the same cell. Since $P(s|_{[n-1]})$ and $P(u|_{[n-1]})$ have at most two rows, the only way $P(u)$ could admit an $n-2/n-1$-double edge is if $n$ is in the second row. Thus, $P(\phi(s))$ has $n$ in the second row and $\phi(s)\in\SYam(\lambda)$.
\\

%
\begin{figure}[h]
\ytableausetup{aligntableaux=center,smalltableaux=on}
\[\begin{array}{l}
\xymatrix{
T_\gamma(s)=  \hspace{-.35in} &\ytableaushort{51,678,234} \ar@{=}[r]^4_5&  \ytableaushort{41,678,235} \ar@{-}[r]^3&  \ytableaushort{31,678,245} \ar@{-}[r]^5&  \ytableaushort{31,578,246}\ar@{=}[r]^4_3& \ytableaushort{41,578,236}& \hspace{-.35in}=T_\gamma(t)\\
}\\
\xymatrix{	
T_\gamma(t)=  \hspace{-.35in} &\ytableaushort{41,578,236}\ar@{=}[r]^3_4 & \ytableaushort{31,578,246} \ar@{-}[r]^2& 
\ytableaushort{21,578,346} \ar@{-}[r]^4 & \ytableaushort{21,478,356} \ar@{=}[r]^2_3& \ytableaushort{31,478,256}  &\hspace{-.35in}=T_\gamma(u)
}\\
\xymatrix{	
T_\gamma(t)=  \hspace{-.35in} &\ytableaushort{41,578,236}\ar@{-}[r]^{7} & \ytableaushort{41,568,237} & \hspace{-.35in}=T_\gamma(v)\\
}
\end{array}\]
 \ytableausetup{aligntableaux=center,boxsize=.3in}
\caption{Specific examples of $s\in\SYam_\gamma((4, 4))$, $t$, $u$, and $v$ from Lemma~\ref{yam to yam induction}, viewed as fillings of $\gamma=(3, 3, 2)$.}
\label{big specific sequence}
\end{figure}
%
\begin{figure}[h]
\vspace{-.5in}
\ytableausetup{aligntableaux=center,boxsize=.3in}
\[ \begin{array}{l}
\scalebox{.92}{
\xymatrix{	
	\raisebox{-.5in}[-.5in]{$P(s)$}&&\raisebox{-.5in}[-.5in]{$P(t)$}\\
	\ytableaushort{{\scriptstyle \lambda_1+1}{ \scriptstyle \lambda_1+2}{}{\scriptstyle n-1}{\scriptstyle n},12{}{}{\scriptstyle \lambda_1-1}{\scriptstyle \lambda_1}}
	 \,\ar@{--}[rr]
	 &&
	 \ytableaushort{{\scriptstyle \lambda_1}{ \scriptstyle \lambda_1+1}{}{\scriptstyle n-1}{\scriptstyle n},12{}{}{\scriptstyle \lambda_1-1}{\scriptstyle n-2}} 
	 }} \vspace{-.4in}
	 \\
\scalebox{.92}{
\xymatrix{	
	 \raisebox{-.5in}[-.5in]{$P(t)$}&&\raisebox{-.5in}[-.5in]{$P(u)$}\\
	  \ytableaushort{{\scriptstyle \lambda_1}{ \scriptstyle \lambda_1+1}{}{\scriptstyle n-1}{\scriptstyle n},12{}{}{\scriptstyle \lambda_1-1}{\scriptstyle n-2}}
	 \,\ar@{--}[rr]
	 && 
	 \ytableaushort{{\scriptstyle \lambda_1-1}{ \scriptstyle \lambda_1}{}{\scriptstyle n-1}{\scriptstyle n},12{}{}{\scriptstyle n-3}{\scriptstyle n-2}}
	 }}\vspace{-.4in}
\\
\scalebox{.92}{
\xymatrix{	
	\raisebox{-.5in}[-.5in]{$P(t)$}&\raisebox{-.5in}[-.5in]{$P(v)$}&& \raisebox{-.5in}[-.5in]{$P(v|_{[n-2]})$}\\
	 \ytableaushort{{\scriptstyle \lambda_1}{ \scriptstyle \lambda_1+1}{}{\scriptstyle n-1}{\scriptstyle n},12{}{}{\scriptstyle \lambda_1-1}{\scriptstyle n-2}}
	 \,\ar@{-}[r]^{n-1}
	 &
	 \ytableaushort{{\scriptstyle \lambda_1}{ \scriptstyle \lambda_1}{}{\scriptstyle n-2}{\scriptstyle n},12{}{}{\scriptstyle \lambda_1-1}{\scriptstyle n-1}}
	 &&
	  \ytableaushort{{\scriptstyle \lambda_1}{ \scriptstyle \lambda_1}{}{\scriptstyle n-2},12{}{}{\scriptstyle \lambda_1-1}}
	 }}
	 \end{array}\]
\caption{The relationships between the standard Young tableaux $P(s)$, $P(t)$, $P(u)$, $P(v)$, and $P(v|_{[n-2]})$.}
\label{sequence of edges}
\end{figure}
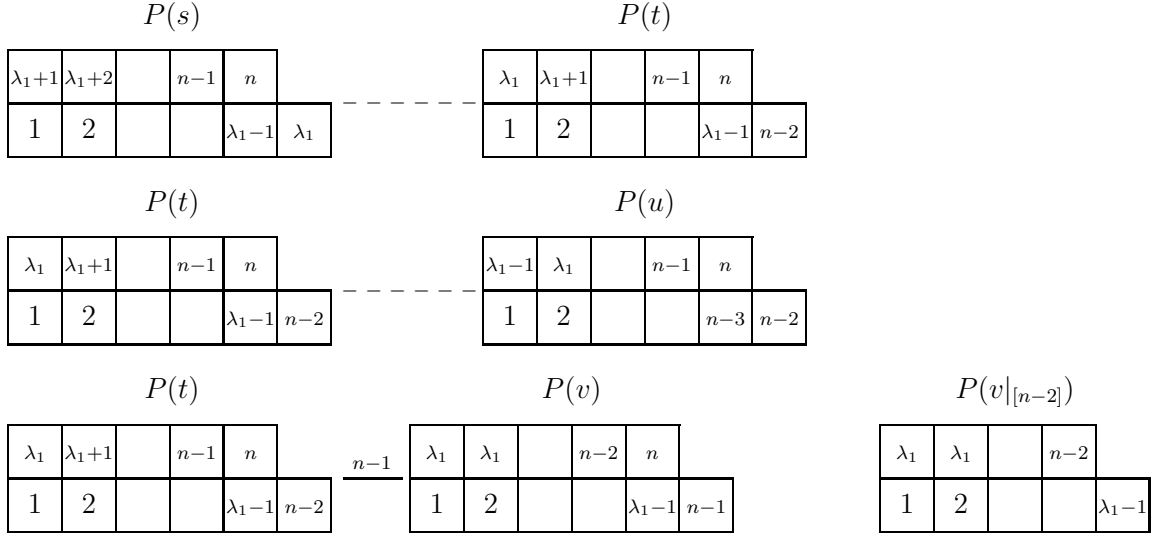

\noindent
\emph{Proof that $\phi(s)$ does not jam $\delta$}:
Because $\lambda$ has two rows, all values of $s$ weakly less than $\lambda_1$ are taken to $1$ in $\unst(\phi(s))$ and all values greater than $\lambda_1$ are taken to 2 in $\unst(\phi(s))$. Applying Definition~\ref{jammed}, it thus suffices to show that the indices of $\phi(s)$ with values weakly less than $\lambda_1$ do not jam $\delta$. We proceed by continuing our analysis of the vertices $t$ and $u$ described above.

Consider the $n-2/n-1$-double edge containing $\phi(u)$. It follows from (\ref{def: D}) that a permutation $\pi$ can only be contained in an $i/i+1$-double edge of $\mh_\delta$ if the indices of $i$ and $i+1$ in $\pi$ are not $\delta$-pistoled. Specifically, $D_i$ and $D_{i+1}$ fix $i+2$ and $i-1$, respectively, and so an $i/i+1$-double edge must act by switching the locations of $i$ and $i+1$. This is only the case when $i$ and $i+1$ are not $\delta$-pistoled. 
In particular, the indices of $n-2$ and $n-1$ in $\phi(u)$ are not $\delta$-pistoled. 
Thus, the indices of $\lambda_1$ and $n-1$ in $\phi(s)$ are not $\delta$-pistoled. 

We still need to show that no index of $\phi(s)$ with value strictly less than $\lambda_1$ can jam $\delta$. Let $v$ be the vertex connected to $t$ by an $n-1$-edge to move $n-1$ into the last index, portrayed via example in Figure~\ref{big specific sequence} and via insertion tableaux in Figure~\ref{sequence of edges}. Applying the above analysis of the pistols of $\gamma$, this $n-1$ edge also acts via $d_i$, swapping $n-2$ and $n-1$. The same must also be true for the $n-1$-edge connecting $\phi(t)$ and $\phi(v)$.
Notice that $v|_{[n-2]}\in\SYam((\lambda_1-1,\lambda_2-1))$, as shown by Figure~\ref{sequence of edges}. Further, $\unst(v|_{[n-2]})$ is the result of removing the furthest right 1 and 2 from $\unst(s)$, so $v|_{[n-2]}\SYam_{\gamma'}(\lambda_1-1, \lambda_2-1)$, where $\gamma^\prime$ is the diagram of $T_\gamma(v)|_{[n-2]}$. By our inductive hypothesis, $\phi(v)|_{[n-2]}$ does not jam $\delta^\prime$, where $\delta^\prime$ is the diagram of $T_\delta(\phi(v))|_{[n-2]}$. Comparing $\phi(s)$ to $\phi(v)$, it follows that no index of $\phi(s)$ with value less than $\lambda_1$ can jam $\delta$. Hence, $\phi(s)$ does not jam $\delta$.
\end{proof}

\begin{lem}\label{yam to yam}
Let $\mg$ and $\mh$ be connected components of $\mh_\gamma$ and $\mh_{\delta}$, respectively. Suppose that there exists an isomorphism $\phi \colon \mg\to \mh$. If $s$ is a vertex of $\mg$ such that $s\in \SYam_\gamma(\lambda)$, then $\phi(s)\in \SYam_{\delta}(\lambda)$.
\end{lem}
\begin{proof}
Let $s$ be as in the statement of the lemma. In the case where $\lambda$ has exactly one row, $\mg$ consists of a single vertex, and the result follows from the fact that isomorphisms preserve signatures. 

Using the partitions with exactly one row as a base case, we may prove the case where $\lambda$ has at most two rows by induction.  Induct on $\lambda_2\geq 1$, assuming the result for one and two row shapes whose second row is smaller than $\lambda_2$. By our inductive hypothesis, $\phi(s)|_{[n-1]}\in\SYam_{\delta'}(\lambda_1, \lambda_2-1)$, where $\delta'$ is the diagram of $T_\delta(\phi(s))|_{[n-1]}$. The inductive step is completed by Lemma~\ref{yam to yam induction}.

Finally, assume that $\lambda$ has more than two rows. Suppose, for the sake of contradiction, that $\phi(s)\notin\SYam_\delta(\lambda)$. In particular, there are some minimal values $i$ and $i+1$ in $\unst(s)$ such that $P(\unst(s))$ has an $i+1$ weakly below the $i^{th}$ row or that satisfy the definition of jamming. Let $I$ be the interval of integers  in $s$ that are sent to $i$ and $i+1$ in $\unst(s)$.
Then $\st(s|_I)_{\gamma'}\in\SYam(\mu)$, where $\mu$ has exactly two rows and $\gamma^\prime$ is the diagram of $T_\gamma(s)|_I$. Further, $\st(\phi(s)|_I)\notin\SYam_{\delta'}(\lambda)$, where $\delta^\prime$ is the diagram of $T_\delta(\phi(s))|_I$.  

The component of $\st(s|_I)$ in $\mh_{\gamma^\prime}$ is isomorphic to the component of $s$ in $\mh_\gamma|_I$. Similarly, the component of $\st(\phi(s)|_I)$ in $\mh_{\delta^\prime}$ is isomorphic to the component of $\phi(s)$ in $\mh_\delta|_I$. Hence, there exists an isomorphism from the component of $\st(s|_I)$ in $\mh_{\gamma^\prime}$ to the component of $\st(\phi(s)|_I)$ in $\mh_{\delta^\prime}$. Appealing to two row case above provides the desired contradiction. Thus $\phi(s)\in\SYam_\delta(\lambda)$.
%
%
\end{proof}

\begin{proof}[{\bf Proof of Theorem~\ref{yam decomposition of graphs}}]

Recall that $\delta$ is a diagram and that $\mg=(V, \sigma, E)$ is a dual equivalence graph such that $\mg$ is a component of $\mh_{\delta}$. By Theorem~\ref{main}, we may further suppose that $\mg \cong \mg_\lambda$, where $\lambda \vdash n$. We need to show that $|V\cap\SYam_\delta(\lambda)|=1$ and that $|V\cap\SYam_\delta(\mu)|=0$  if $\lambda \neq \mu$.

In order to apply Lemma~\ref{yam to yam}, we first need to show that $\mg_\lambda$ is isomorphic to a component of a well-behaved $\mh_\gamma$.  By conflating standard Young tableaux in $\mg_\lambda$ with their row reading words, we may consider $\mg_\lambda$ as a component $\mc$ of $\mh_{\gamma}$, where $\gamma$ is the subset of the vertical axis $\{(0,i)\in \zz \times \zz\colon 0\leq i< n \}$. Notice that the unique standardized Yamanouchi word in $\mc$ is $\rw(U_\lambda)$ and that $\SYam_{\gamma}(\lambda)=\SYam(\lambda)$. By composing isomorphisms, we may find an isomorphism $\phi\colon \mg \to \mc$. By Lemma~\ref{yam to yam}, $\phi$ acts as a bijection between the set of standardized Yamanouchi words in $\mg$ that do not jam $\delta$ and the set of standardized Yamanouchi words in $\mc$ that do not jam $\gamma$. Thus, there is a unique permutation in $V\cap \SYam_\delta(\lambda)$ corresponding to $\phi^{-1}(\rw(U_\lambda))$. Hence, $|V\cap 
\SYam_\delta(\lambda)|=1$. Since $\SYam(\lambda)$ and $\SYam(\mu)$ are disjoint when $\lambda \neq \mu$, it follows that $V\cap \SYam_\delta(\mu)=\emptyset$ when $\lambda\neq\mu$. 
\end{proof}

\begin{cor}\label{yam jam functions}
If $\mg=(V, \sigma, E)$ is a dual equivalence graph contained in  $\mh_{\delta}$, then
\begin{equation}
\sum_{v\in V} F_{\sigma(v)}(X) = \sum_{\lambda\vdash n} |V\cap \SYam_\delta(\lambda)| \cdot s_\lambda
\end{equation}
\end{cor}
\begin{proof}
This is an immediate consequence of Theorem~\ref{yam decomposition of graphs} and (\ref{schurs}). 
\end{proof}

\begin{proof}[{\bf Proof of Theorem \ref{Hall Yam classification}}]
Expressing $\Des$ and $\HLdelta$ as in (\ref{P by graph}), the result follows from 
Theorems~\ref{Halls are DEGs}, 
Corollary~\ref{yam jam functions}, and the fact that the $\maj_\delta$ statistic is constant on components of $\mP_\delta$.
\end{proof}

\begin{remark}
\textup{
\begin{enumerate}
\item
For each partition $\lambda$ and diagram $\delta$, there exists a set $Y_\delta(\lambda)$ defined as the intersection of $\SYam_\delta(\lambda)$ with the set of permutations of length $|\delta|$ whose component in $\mh_\delta$ is a dual equivalence graph. That is, if $\mg=(V, \sigma, E)$ is a component of $\mh_{\delta}$, then $|V\cap Y_\delta(\lambda)|=1$ if  $\mg\cong\mg_\lambda$, and $|V\cap S_\delta(\lambda)|=0$ otherwise. Finding a more direct way to generate $Y_\delta(\lambda)$, however, is an open problem.
%
\item
We may use Theorem~\ref{Halls are DEGs} to find two related families of dual equivalence graphs in $\mh_\delta$. Consider the graphs obtained by reversing the values in the vertex sets of $\mP_\delta$ and $\mr_{\gamma,\delta}$, sending $i$ to $|\delta|+1-i$. The set of edge labels is similarly reversed, while each $\sigma_i$ is sent to $-\sigma_{|\delta|-i}$. This is equivalent to multiplying each signature of the `color reversal' by $-1$,  as described in \cite[Cor.~3.8]{DEG+LLT}. Thus, it necessarily sends dual equivalence graphs in $\mh_\delta$ to other dual equivalence graphs in $\mh_\delta$ and sends each associated $s_\lambda$ to $s_{\lambda'}$. We may then apply Corollary~\ref{yam jam functions} to describe the Schur expansions of the related symmetric functions. Combinatorially, the reversing of values described above is an involution that maps the set of fillings that have no inversion pairs or triples onto the set of fillings that have every possible inversion pair and triple. That is, we are restricting our attention to permutations $\pi$ that achieve the maximal $\inv_\delta(\pi)$, denoted $m(\delta)$. The associated symmetric function may also be computed by replacing $s_\lambda$ with $s_{\lambda^\prime}$, and replacing $\maj_\delta$ with $\operatorname{comaj}_\delta$, which is the result of subtracting maj from the maximal value of $\maj_\delta(\pi)$. 
That is,
\[
\Macdelta\big|_{q^{m(\delta)}}
=\sum_{\lambda \vdash |\delta|} \;
\sum_{{w \in \Yam_{\delta}(\lambda)}\atop \inv_\delta(w)=m(\delta) }t^{\maj_\delta(w)} s_{\lambda}
=\sum_{\lambda \vdash |\delta|} \;
\sum_{{w \in \Yam_{\delta}(\lambda)}\atop \inv_\delta(w)=0 }t^{\operatorname{comaj}_\delta(w)} s_{\lambda^\prime}.
\]
\end{enumerate} }
\end{remark}

\subsection{Related Conjectures}

The above analysis lends itself to an interesting conjecture about the Schur expansion of the quasisymmetric function associated with any graph comprised of components of $\mh_{\delta}$.

\begin{conjecture}\label{f tau}
Given any diagram $\delta$ with $n$ cells, there exists an injective function $f_\delta\colon \SYam(n)\hookrightarrow S_n$ fixing $\SYam_\delta(n)$ and preserving $\sigma$ such that for any component $\mc=(V,\sigma,E)$ of $\mh_{\delta}$, 
\begin{equation}
\sum_{v\in V} F_{\sigma(v)}(X)=\sum_{\lambda\vdash n} 
|\{\pi \in \SYam(\lambda)\colon f_\delta(\pi)\in V\}|\cdot s_\lambda
\end{equation}
\end{conjecture}

\begin{conjecture}[Corollary of Conjecture~\ref{f tau}]\label{f tau cor}
Given any diagram $\delta$ with $n$ cells and function $f_\delta$ as in Conjecture~\ref{f tau},
\[
\Macdelta = \sum_{\lambda\vdash n} \,
\sum_{\pi \in \SYam(\lambda)} q^{\inv_\delta(f_\delta(\pi))}t^{\maj_\delta(f_\delta(\pi))}s_\lambda.
\]
\end{conjecture}

\noindent
Conjecture~\ref{f tau cor} has been explicitly checked when $\delta$ is a partition shape of size at most seven. It should be mentioned, however, that $f_\delta$ was defined in an ad hoc fashion for each new $\delta$.

\section{Further Applications to Symmetric Functions}\label{Applications to Symmetric Functions}

\subsection{More analysis of $\HL$ and $\MacStraight$}\label{further applications}
We can now explicitly answer the question of Garsia mentioned in Section~\ref{intro}. We also provide the analogous result for Macdonald polynomials.

\begin{cor}\label{Hall Yam cor classification}
Given $\mu\vdash n$, the following equality holds if and only if $\mu$ does not contain $(3, 3, 3)$ as a subdiagram.
\begin{align}\label{HL Yam cor equation}
\HL =& \sum_{\lambda \vdash n} \;
\sum_{w \in \Yam(\lambda) \atop \inv_\mu(w)=0} t^{\maj_\mu(w)}s_\lambda.
\end{align}
\end{cor}
\begin{proof}
First assume that $(3, 3, 3)$ is a subdiagram of $\mu$. In light of Theorem~\ref{Hall Yam classification}, it suffices to show that there exists $w \in\Yam(\lambda)$ for some $\lambda\vdash |\mu|$ such that $\inv_\mu(w)=0$ and $w$ jams $\mu$. We may explicitly demonstrate the desired $w$ by placing 1's in all cells of $\mu$ except the first cell of the second row and the first three cells of the third row. Now fill the four remaining cells with 3232 in row reading order, as in Figure~\ref{unpaired Yam}, and then define $w$ as the row reading word of this filling. Thus, $\inv_\mu(w)=0$, and the index of the last 2 of $w$ jams the pistol starting at the first cell of the third row from the bottom.

Now suppose that $(3,3,3)$ is not a subdiagram of $\mu$. We need to show that  
\begin{equation}\label{yams equal}
\{w\in \Yam_\mu(\lambda) \colon \inv_\mu(w)=0\} = \{w\in \Yam(\lambda) \colon \inv_\mu(w) =0\}
\end{equation}
 to apply Theorem~\ref{Hall Yam classification}. Let $w \in \Yam(\lambda)$ be the row reading word of a filling $T$ of $\mu$ such that $\inv(T)=0$. Since a Yamanouchi word must end in a 1, and the bottom row of $T$ must be weakly increasing to avoid inversion pairs, the bottom row must be all 1's. Similarly focusing on the construction of Yamanouchi words and the description of inversion triples in Figure~\ref{triple},  it is readily shown that the second row starts with some number of 2's followed by all 1's. Thus, $w$ cannot jam any of the pistols contained in the bottom two rows. Now consider any pistol that ends before the bottom row. Since $\mu$ does not contain $(3,3,3)$, any such pistol has at most three cells. In a Yamanouchi word, the index of the $j^{th}$ $i$ differs by at least three from the index of the $j+1^{th}$ $i+1$, and so cannot share a pistol containing less than four cells. Hence, no pistol that ends before the bottom row can be jammed by an index of $w$. That is, (\ref{yams equal}) holds.
\end{proof}

\begin{figure}[H] 
   \begin{center} 
   \ytableausetup{smalltableaux}
\ytableaushort{11,323,211,1111}
   \end{center}
  \caption{A filling with row reading word $w \in\SYam(\lambda)$ such that $\inv_\mu(w)=0$ and $w$ jams $\mu$.}
 \label{unpaired Yam}
\end{figure}

\begin{prop}\label{Mac Yam classification}
Given $\mu\vdash n$, the following equality holds if and only if $\mu$ does not contain $(4)$ or $(3, 3)$ as a subdiagram.
\begin{equation}\label{Mac Yam equation}
\MacStraight = \sum_{\lambda \vdash n} \;
\sum_{w \in \Yam(\lambda)}  q^{\inv_\mu(w)}t^{\maj_\mu(w)}s_\lambda.
\end{equation}
\end{prop}

\begin{proof}
If we assume that $\mu$ does not contain $(4)$ or $(3, 3)$ as a subdiagram, the result is given by Lemma~\ref{3 cell pistols}. It then suffices to assume that $\mu$ contains (4) or (3, 3) and show that Equation (\ref{Mac Yam equation}) does not hold. 

We proceed by considering  the coefficients of $q^2t^0$. Focusing on the right hand side of Equation (\ref{Mac Yam equation}), consider a Yamanouchi word $w$ such that $\inv_\mu(w)=2$ and $\maj_\mu(w)=0$. In particular, $\maj_\mu(w)=0$ forces the columns of $T_\mu(w)$ to be weakly increasing when read downward. 
The bottom row of $T_\mu(w)$ must contribute at most two inversions, so $w$ must end in 1\ldots 111, 1\ldots 121, or 1\ldots 1211. That is, the bottom row can have at most one 2.
Applying the fact that all columns are weakly increasing, all but one column must be all 1's. 

We now consider each of these possibilities for the bottom row separately.
If the bottom row is all 1's, then every column must be all 1's, and so there cannot be any inversions. If the bottom row is $1\ldots 121$ or $1\ldots 1211$ and there is an $x > 1$ above the 2 in the bottom row, then there are at least two inversion triples containing $x$ in the first case and at least one in the second case. Both situations force $\inv_\mu(w)>2$. Hence, any values greater than one must occur in the bottom row. Because $\inv_\mu(w)=2$, we are left with only the filling containing all 1's except for a bottom row filled by $1\ldots 1211$. The conclusion of this analysis is that

\begin{equation}
\sum_{\lambda \vdash n} \;
\sum_{w \in \Yam(\lambda)}  q^{\inv_\mu(w)}t^{\maj_\mu(w)}s_\lambda\big|_{q^2t^0}= s_{(n-1,1)}.
\end{equation}

Now consider the coefficient of $q^2t^0$ in $\MacStraight$, as described in (\ref{Macdonalds}), 
 when $\mu$ contains $(4)$ or $(3, 3)$.  By letting $\pi$ be  the standardization of  $w = 1\dots 1\,1\,2\,1\,3\,2$, we have $\maj_\mu(\pi)=0$ and $\inv_\mu(\pi)=2$ (see Figure~\ref{NonYam}). Thus, $q^2t^0F_{\sigma(\pi)}$ has a positive coefficient in the expansion of $\MacStraight$ into fundamental quasisymmetric functions. However, $F_{\sigma(\pi)}$ has coefficient 0 in the expansion of $s_{(n-1,1)}$. This is clear because $\sigma(\pi)= +\cdots++-+-$, but all fundamental quasisymmetric functions that contribute to $s_{(n-1,1)}$ have exactly one minus sign. Hence, there must be some term of $\MacStraight$ of the form $q^2t^0s_\lambda$, where $\lambda \neq (n-1, 1)$. 
Therefore, (\ref{Mac Yam equation}) cannot hold if $\mu$ contains (4) or (3,3) as a subdiagram, completing our proof.
\end{proof}
%
\begin{figure}[H]
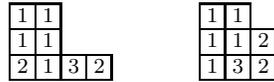
 
   \begin{center} 
   \ytableausetup{smalltableaux}
\ytableaushort{11,11, 2132}
\hspace{.4in}
\ytableaushort{11,112,132}
   \end{center}
  \caption{Fillings by $w=1\ldots112132$ with $\inv_\mu(w)=2$ and $\maj_\mu(w)=0$.}
 \label{NonYam}
\end{figure}

\subsection{Further Analysis of $\Des$}

In this section we give a method for quickly finding the leading term of the expansion of $\Des$. We begin by using the relationship between $\gamma$ and $\delta$ to give a characterization of when $\Des=0$. To do so, we will need the following definition.

\begin{dfn}\label{realizable}
\textup{
Given any diagrams $\gamma$ and $\delta$ such that $\gamma \subset \delta$, we refer to $\gamma$ as a \emph{realizable descent set} of $\delta$ if the following hold.
\begin{enumerate}
\item If $(x,y)\in \gamma$, then $(x,y-1)\in \delta$.
\item If  $x_1<x_2$ are any integers and $I$ is any integer interval such that $(x_1,\max(I)), (x_1,\min(I))\notin \gamma$, and $(x_1, I\setminus \{\min(I), \max(I)\})\subset \gamma$,  then $(x_2, I\setminus\{\min(I)\})\not\subset \gamma$.
\end{enumerate}
}
\end{dfn}

Examples of Definition~\ref{realizable} can be found in Figure~\ref{realizable figure}. This definition was chosen to aid in the statement of the following proposition.

\begin{figure}[H]
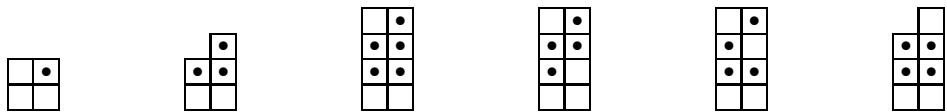

   \begin{center} 
   \ytableausetup{smalltableaux, aligntableaux=bottom}
\ytableaushort{{}\bullet,{}{},} \hspace{.6in}
\ytableaushort{\none\bullet, \bullet\bullet,{}{}} \hspace{.6in}
\ytableaushort{{}\bullet, \bullet\bullet,\bullet\bullet,{}{}} \hspace{.6in}
\ytableaushort{{}\bullet, \bullet\bullet,\bullet{},{}{}} \hspace{.6in}
\ytableaushort{{}\bullet, \bullet{},\bullet\bullet,{}{}} \hspace{.6in}
\ytableaushort{\none{}, \bullet\bullet,\bullet\bullet,{}{}}
 \end{center}
\caption{Diagrams with bullets representing $\gamma$ and boxes representing $\delta$. From the left, three examples where $\gamma$ is not a realizable descent set of $\delta$, then three examples where $\gamma$ is a realizable descent set of $\delta$.}\label{realizable figure}
\end{figure}

\begin{prop}\label{Des zero}
Given any two diagrams $\gamma$ and $\delta$ such that $\gamma\subset\delta$, then $\Des\neq 0$ if and only if $\gamma$ is a realizable descent set of $\delta$.
\end{prop}
The proof of Proposition~\ref{Des zero} is postponed until the end of this section.

\begin{dfn}\label{leading Yamanouchi}
\textup{Given diagrams $\gamma$ and $\delta$ such that $\gamma$ is a realizable descent set of $\delta$, define the \emph{leading Yamanouchi word} of $\Des$, denoted $w_{\gamma, \delta}$, as the row reading word of the filling of $\delta$ achieved by placing a 1 in every cell of $\delta\setminus \gamma$ and then placing values in the rows of $\gamma$ from bottom to top, filling each cell with one plus the value in the cell immediately below it in $\delta$.}
\end{dfn}

\noindent
See Figure~\ref{leading Yam} for an example of Definition~\ref{leading Yamanouchi}. Notice that $w_{\gamma, \delta}$ is indeed a Yamanouchi word. We can then use $w_{\gamma,\delta}$ to provide the leading term in the expansion of $\Des$ into Schur functions.

\begin{figure}
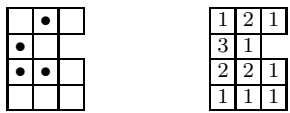

\begin{center}
\ytableaushort{{}\bullet{}, \bullet{}\none,\bullet\bullet{},{}{}{}} \hspace{.6in}
\ytableaushort{121, 31\none,221,111}
\end{center}
\caption{At left, a diagram $\delta$ with bullets in cells of a realizable subset $\gamma$. At right, a filling whose row reading word is the leading Yamanouchi word $w_{\gamma, \delta}=12131221111$.}\label{leading Yam}
\end{figure}

\begin{prop}\label{Des leading term}
Given diagrams $\gamma$ and $\delta$ such that $\gamma$ is a realizable descent set of $\delta$, let $\Des=\sum_\lambda c_\lambda s_\lambda$ for some nonzero integers $c_\lambda$, and let $w_{\gamma, \delta}\in \Yam(\mu)$, then 
\begin{enumerate}
\item $c_\lambda=1$ if $\lambda=\mu$,
\item $c_\lambda = 0$ if $\lambda>\mu$ in lexicographic order.
\end{enumerate}
\end{prop}
\begin{proof}
Let   $w_{\gamma, \delta}$ be as described in the proposition. We begin by showing that $T=T_\delta(w_{\gamma,\delta})$ has $\Descent(T)=\gamma$ and $\inv(T)=0$. The fact that $\Descent(T)=\gamma$ is a result of the definition of $w_{\gamma,\delta}$. We thus need to show that $\inv(T)=0$. 

First consider a triple of $T$ consisting of cells $c$, $d$, and $e$, in row reading order. Notice that if $c$ is a descent of $T$, then the value in $c$ is one greater than the value in $e$. It then follows that cells $c$, $d$, and $e$ cannot form an inversion triple, regardless of the value in $d$. Now suppose that $c$ is not a descent of $T$. Then the value in $c$ is a 1, so it suffices to show that the value in $d$ is weakly less than the value in $e$. If the value in $d$ is 1, we are done. We may then assume that $d$ is a descent. By the definition of $w_{\gamma, \delta}$, it follows that the value in $d$ is one greater than the number of consecutive descents of $T$  occurring in the cells weakly below $d$. Appealing to the definition of realizable descent set, there must be at least as many consecutive descents of $T$  starting at $e$ and continuing down. Thus, the value in $d$ is weakly less than the value in $e$. That is, the cells $c$, $d$, and $e$ do not form an inversion triple.


Next consider two cells $c$ and $d$, in row reading order, that could form an inversion pair. If $d$ is in the row below $c$, then the same reasoning as above shows that $d$ must be weakly greater than $c$. If $c$ and $d$ are in the same row, then the cell below $c$ cannot be in $\delta$. Thus $c$ is not a descent of $T$ and so has value 1 in $T$. In particular, $c$ and $d$ do not form an inversion pair. Therefore, $\inv(T)=0$.

Notice that the values in each cell of $T$ are as small as possible while respecting $\Descent(T)=\gamma$. Thus, there cannot be any other filling with  row reading word $w \in\SYam(\lambda)$, $\inv_\delta(w)=0$, and $\Descent_\delta(w)=\gamma$, where $\lambda \geq \mu$ in lexicographic order. Appealing to Theorem~\ref{Hall Yam classification} completes the proof.
\end{proof}



\begin{proof}[{\bf Proof of Proposition~\ref{Des zero}}]
We wish to show that $\gamma$ is a realizable descent set of $\delta$ if and only if $\Des\neq0$. First assume that $\gamma$ is not a realizable descent set of $\delta$. Notice that if $\gamma$ does not satisfy Part~1 of Definition~\ref{realizable}, then there are no fillings of $\delta$ with descent set $\gamma$. Thus, $\Des=0$. Next,
suppose that $T$ is a filling of $\delta$ with $\Descent(T)=\gamma$ such that $\gamma$ satisfies Part~1 but not Part~2 of Definition~\ref{realizable}. It suffices to show that $\inv(T)\neq 0$. We suppose, for the sake of contradiction, that $\inv(T)= 0$.

Choose $\{x_1,x_2\}\times I$ violating Part~2 of Definition~\ref{realizable}. 
Label the values of $T$ in this rectangle by $a, b, c$, and so on, in row reading order. Here, we will call the upper left value $a$, whether or not the value exists in $T$. Regardless of existence in $\delta$, the cell containing $a$ is not in $\gamma$. See Figure~\ref{realizable rectangle} for an illustration.
  
 Because $a$ is not in a descent of $T$, and $T$ has no inversion triples or inversion pairs, it follows that $a\leq b \leq c$. Here, we have $b<c$ if $a$ does not exist. Also, $b>d$ because $b$ is in a descent of $T$, and so $c>d$. Assuming that $c$ is a descent of $T$, it then follows that $c>e\geq d>f$, since $T$ has no inversions and $d$ is in a descent of $T$. In particular, notice that $e>f$. Continuing in this fashion recursively, it follows that the value in the bottom left corner is greater than the value in the bottom right corner. However, the bottom left corner is not a descent of $T$, and so the fact that its value is greater than the value to its right guarantees an inversion triple or inversion pair in $T$. Hence, $\inv(T)\neq 0$, as desired. Therefore, $\Des=0$ whenever $\gamma$ is not a realizable descent set of $\delta$.

\begin{figure}[H]
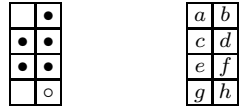

   \begin{center} 
   \ytableausetup{smalltableaux, aligntableaux=center}
\ytableaushort{{}\bullet, \bullet\bullet,\bullet\bullet,{}\circ} \hspace{.6in}
\ytableaushort{ab,cd,ef,gh}
 \end{center}
\caption{At left, a rectangle violating Part~2 of Definition~\ref{realizable} with $\gamma$ represented by bullets and possibly the open circle. At right, the values of $T$ in this rectangle.}\label{realizable rectangle}
\end{figure}

We still need to consider the case where $\gamma$ is a realizable descent set of $\delta$ and show that $\Des\neq 0$ in this case. By Proposition~\ref{Des leading term}, the Schur expansion of $\Des$ has a nonzero term, and so $\Des \neq 0$, completing the proof.
\end{proof}


\subsection*{Acknowledgements}

We would like to thank Adriano Garsia, both for the motivation to work on this problem, as well as his continued conversation while formulating this paper. We also wish to thank Avinash Dalal, James Haglund, Angela Hicks,  Emily Leven, Monty McGovern, Isabella Novik, Jennifer Morse, Brendan Pawlowski, and Mike Zabrocki for inspiring conversations on this topic. Particular thanks belong to Sara Billey for her invaluable guidance and editorial help throughout the process of creating this paper. Finally, thank you to the anonymous referees for their time and insight towards improving this paper.

\bibliographystyle{abbrvnat}
\bibliography{BIBDatabase}

\end{document}